\documentclass [11pt,a4paper]{article}
\usepackage[english]{babel}
\usepackage[T1]{fontenc}
\usepackage[latin1]{inputenc}
\usepackage[a4paper,left=2.5cm,right=2.5cm]{geometry}

\usepackage{amssymb,amsmath,amsthm,multicol,tabularx,makeidx,bbm,mathrsfs}
\usepackage{fancyhdr,graphicx,color,listings}
\usepackage[outdir=./]{epstopdf}

\usepackage{enumitem}
\usepackage{nccmath}

\theoremstyle{definition}
\newtheorem{theorem}{Theorem}[section]

\newtheorem{lemma}[theorem]{Lemma}
\newtheorem{definition}[theorem]{Definition}
\newtheorem{remark}[theorem]{Remark}
\newtheorem{corollary}[theorem]{Corollary}

\newcommand*\diff{\mathop{}\!\mathrm{d}}

\theoremstyle{definition}
\newtheorem*{res}{Main result}
\newtheorem*{plan}{Plan}

\makeatletter
\newcommand\ackname{Acknowledgements}
\if@titlepage
   \newenvironment{acknowledgements}{%
       \titlepage
       \null\vfil
       \@beginparpenalty\@lowpenalty
       \begin{center}%
         \bfseries \ackname
         \@endparpenalty\@M
       \end{center}}%
      {\par\vfil\null\endtitlepage}
\else
   \newenvironment{acknowledgements}{%
       \if@twocolumn
         \section*{\abstractname}%
       \else
         \small
         \begin{center}%
           {\bfseries \ackname\vspace{-.5em}\vspace{\z@}}%
         \end{center}%
         \quotation
       \fi}
       {\if@twocolumn\else\endquotation\fi}
\fi
\makeatother

\title{On the convergence of stochastic transport equations to a deterministic parabolic one}
\author{Lucio Galeati}
\date{\small{Institute of Applied Mathematics\\ University of Bonn, Germany}\\[1.2ex] lucio.galeati@iam.uni-bonn.de\\[1.5ex] \today}

\begin{document}
\maketitle
\thispagestyle{empty}
\begin{abstract}
\noindent A stochastic transport linear equation (STLE) with multiplicative space-time dependent noise is studied. It is shown that, under suitable assumptions on the noise, a multiplicative renormalization leads to convergence of the solutions of STLE to the solution of a deterministic parabolic equation. Existence and uniqueness for STLE are also discussed. Our method works in dimension $d\geq 2$; the case $d=1$ is also investigated but no conclusive answer is obtained. 
\end{abstract}
\section{Introduction}\label{section 1 - introduction}

%
%

Throughout this paper, we consider a stochastic transport linear equation of the form
\begin{equation}\label{sec 1 - STLE in compact Stratonovich form}\tag{STLE}
\diff u = b\cdot\nabla u \diff t + \circ \diff W\cdot \nabla u,
\end{equation}
where $b=b(t,x)$ is a given deterministic function and $W=W(t,x)$ is a space-time dependent noise of the form
\begin{equation}\label{sec 1 - explicit description of the noise}
W(t,x)=\sum_k \sigma_k(x)W_k(t).
\end{equation}
Here $\sigma_k$ are smooth, divergence free, mean zero vector fields, $\{W_k\}_k$ are independent standard Brownian motions and the index $k$ might range on an infinite (countable) set; by \eqref{sec 1 - STLE in compact Stratonovich form} we mean more explicitly the identity
\begin{equation}\label{sec 1 - STLE in non compact Stratonovich form}
\diff u = b\cdot \nabla u \diff t + \sum_k \sigma_k\cdot\nabla u \circ \diff W_k,
\end{equation}
where $\circ$ denotes Stratonovich integral. Let us explain the reasons for studying such equation.

%
%

In the case of space-independent noise, it has been shown in recent years, starting with \cite{FlaGub}, that equation \eqref{sec 1 - STLE in compact Stratonovich form} is well posed under much weaker assumptions on $b$ than its deterministic counterpart (i.e. with $W=0$), for which essentially sharp condition are given by \cite{DiPLio}, \cite{Amb}. There is now an extensive literature on the topic of regularization by noise for transport equations, see the review \cite{Fla2} and the references in \cite{BecFla}. However, from the modelling point of view, space-independent noise is too simple, since formally the characteristics associated to \eqref{sec 1 - STLE in compact Stratonovich form} are given by
\begin{equation}\nonumber
\diff X_t = -b(t,X_t) \diff t - \diff W_t.
\end{equation}
Namely, if we interpret $u$ as an ensemble of ideal particles, the addition of such a multiplicative Stratonovich noise corresponds at the Lagrangian level to non interacting particles being transported by a drift $b$ as well as a random, space independent noise $W$. There are several models, especially those arising in turbulence (see \cite{Cha} and the discussion in the introduction of \cite{CogFla}), in which it seems more reasonable to consider all the particles to be subject to the same space-dependent, environmental noise $W$, which is randomly evolving over time and is not influenced by the particles; $W$ may be interpreted as an incompressible fluid in which the particles are immersed. The formal Lagrangian description of \eqref{sec 1 - STLE in compact Stratonovich form} is
\begin{equation}\label{sec 1 - stochastic characteristics associated, compact Stratonovich form}
\diff X_t = -b(t,X_t)\diff t - \circ \diff W(t,X_t),
\end{equation}
where the above equation is meaningful once we consider $W$ given by \eqref{sec 1 - explicit description of the noise} and we explicit the series.

Another reason to consider a more structured noise is given by the fact that, in the case of nonlinear transport equations, explicit examples in which a space-independent noise doesn't regularize are known, see for instance Section 4.1 of \cite{Fla}; instead a sufficiently structured, space-dependent noise can provide a partial regularization by avoiding coalescence of particles, as in \cite{DFV}, \cite{FlaGub2}.

Finally, if we expect the paradigm \textquotedblleft the rougher the noise, the better the regularization\textquotedblright\ to hold, as it has been observed frequently in regularization by noise phenomena, it is worth to investigate the effect on equation \eqref{sec 1 - STLE in compact Stratonovich form} of a noise $W$ which has poor regularity in space.

Specifically, the main goal of this work is not to investigate well posedness of \eqref{sec 1 - STLE in compact Stratonovich form}, but rather to understand what happens when the space regularity of $W$ is so weak that it's not clear how to give meaning to \eqref{sec 1 - STLE in compact Stratonovich form} anymore. Indeed, when one writes the corresponding It\^o formulation of \eqref{sec 1 - STLE in compact Stratonovich form}, the It\^o-Stratonovich corrector appearing is finite only if $W$ satisfies a condition of the form
\begin{equation}\label{sec 1 - condition on the regularity of the noise}
\mathbb{E}\Big[\vert W(1,\cdot)\vert_{L^2}^2\Big]<\infty.
\end{equation}
In particular, if the above condition doesn't hold, typically the corrector will be of the form \textquotedblleft $+\infty\,\Delta u$\textquotedblright\ and therefore heuristically one would expect the solution to istantaneously dissipate and become constant, independently of the initial data. A rigorous proof of this assertion, by means of a Galerkin approximation, has been given in a specific case in \cite[Theorem 1.3]{FlaLuo}, but the technique applied there seems sufficiently robust to be generalized to this setting as well. It turns out that, in order to obtain a non trivial limit when we consider solutions of \eqref{sec 1 - STLE in compact Stratonovich form} for a sequence of noises $W^N$ whose $L^2$-norm is exploding as $N\to\infty$, a suitable sequence of multiplicative coefficients $\varepsilon^N$ must be introduced. In order to explain better what we mean and to give a rough statement of the main result, we give a  brief description of the setting in which we study \eqref{sec 1 - STLE in compact Stratonovich form}. More details will be given in the next section.

%
%

We consider everything to be defined on the $d$-dimensional torus $\mathbb{T}^d = \mathbb{R}^d/(2\pi\mathbb{Z}^d)$ with periodic boundary condition, $d\geq 2$, with suitable assumptions on $b$.
We denote by $\mathcal{H}$ the closed subspace of $L^2(\mathbb{T}^d;\mathbb{R}^d)$ given by divergence free, mean-zero functions (see Section \ref{subsection 2.1 - notation and functional setting} for the exact definition).

We fix an a priori given filtered probability space $(\Omega,\mathcal{F},\mathcal{F}_t,\mathbb{P})$ on which an $\mathcal{H}$-cylindrical $\mathcal{F}_t$-Wiener process $\widetilde{W}$ is defined, see \cite{DaP}. We apply to $\widetilde{W}$ a Fourier multiplier $\Theta$ such that $W:= \Theta\widetilde{W}$ satisfies \eqref{sec 1 - condition on the regularity of the noise}. We consider, for this choice of $W$, the Cauchy problem given by \eqref{sec 1 - STLE in compact Stratonovich form} together with a deterministic initial condition $u_0\in L^2(\mathbb{T}^d)$; we are interested in energy solutions $u$, namely $\mathcal{F}_t$-progressively measurable processes, with weakly continuous paths, for which equation \eqref{sec 1 - STLE in compact Stratonovich form} is satisfied when interpreted in an analitically weak sense, i.e. testing against smooth functions, and a suitable energy inequality holds.
We stress that we consider $u$ to be a strong solution in the probabilistic sense; we can vary $W$ by considering different choices of $\Theta$, but the probability space and $\widetilde{W}$ are fixed and a priori given. The main result can then be loosely stated as follows.

\begin{res}\label{result sec 1 - introductory formulation of main theorem}
Assume that $b$ satisfies suitable conditions together with the following assumption:
\begin{itemize}
\item[(UN)] (Uniqueness for the parabolic limit equation) $b$  is such that, for any $\nu>0$, uniqueness holds in the class of weak $L^\infty(0,T;L^2(\mathbb{T}^d))$ solutions of the Cauchy problem
\begin{equation}\label{sec 1 - parabolic limit problem}
\begin{cases} \partial_t u = \nu\Delta u + b\cdot\nabla u\\ u(0)=u_0 \end{cases}.
\end{equation}
\end{itemize}
Then for any $\nu>0$, there exists a class of sequences of Fourier multipliers $\Theta^N$ and of constants $\varepsilon^N$, with $\varepsilon^N$ depending only on $\nu$ and $\Theta^N$ for each $N$, such that, denoting $W^N=\Theta^N\widetilde{W}$, for any $u_0\in L^2(\mathbb{T}^d)$, any sequence of energy solutions $u^N$ of the STLEs
\begin{equation}\label{sec 1 - equations for the approximating u^N}\begin{cases}
\diff u^N = b\cdot\nabla u^N\, \diff t + \sqrt{\varepsilon^N}\circ \diff W^N\cdot \nabla u^N\\
u^N(0)=u_0
\end{cases}\end{equation}
converges in probability (in a suitable topology) as $N\to\infty$ to the unique deterministic solution $u$ of the parabolic equation \eqref{sec 1 - parabolic limit problem}.
\end{res}

A more precise statement and the proof will be given in Section \ref{section 3 - proof of the main result}; let us comment some of the features of the result.

\begin{itemize}
\item[i)] The statement is formulated in the spirit of a multiplicative renormalization: the sequence $\varepsilon^N$ depends on the chosen $\Theta^N$, but the limit does not, up to the arbitrary choice of a one dimensional parameter $\nu>0$. However, as will be discussed in Section \ref{section 3 - proof of the main result}, this is not a real renormalization due to the presence of some degeneracy: while we need to impose some conditions on $\Theta^N$, these do not  imply uniqueness of the limit of $\Theta^N \widetilde{W}$ and explicit examples of choices leading to different limits, for which the above statement holds, can be given. In a sense, the result is more similar to a weak law of large numbers, as will be discussed in Section \ref{section 3 - proof of the main result}.

\item[ii)] The statement provides a sequence of solutions of stochastic transport equations converging to a deterministic parabolic equation. This is rather surprising, not only for the transition from a stochastic problem to a deterministic one, but also for the change in the nature of the equation. The original STLEs are hyperbolic: whenever $W$ is regular enough, they can be solved explicitly by means of the stochastic flow associated to the characteristics \eqref{sec 1 - stochastic characteristics associated, compact Stratonovich form}; in particular the solutions don't have in general better regularity than the initial data, at least at the level of trajectories. However, when considering the corresponding It\^o formulation, the It\^o-Stratonovich corrector gives rise to a Laplacian. It was intuited in \cite{FlaGub} that equation \eqref{sec 1 - STLE in compact Stratonovich form} has some parabolic features at the mean level; this has become clear in \cite{BecFla}.

\item[iii)] The statements holds for \textit{any} sequence of energy solutions of \eqref{sec 1 - equations for the approximating u^N}, even when uniqueness is not known; existence of energy solutions can be shown under suitable assumptions on $b$. We only need well posedness for the limit problem \eqref{sec 1 - parabolic limit problem} and that's why we require (UN) to hold. In general (UN) is satisfied under very mild assumptions on $b$, much weaker than those required for the associated deterministic transport equation to be well posed. This suggests the possibility to obtain uniqueness for the STLE under the same assumption (UN); in this direction, see the results given in \cite{Mau}, \cite{BecFla} and the references therein.

\item[iv)] From the modelling, perturbative viewpoint, the result could be interpreted in this way: when a system of particles transported by a drift $b$ is subject to an environmental background noise which is very irregular but of very small intensity, in the ideal limit such a disturbance is correctly modelled by a diffusive term $\nu\Delta$. This also gives an interesting link between different selection principles for ill posed transport equations, since it hints to the fact that a vanishing viscosity limit and certain types of zero noise limits should behave similarly; observe however that this is not true in general, since in the setting of space-independent noise, examples of transport equations for which the zero noise limit and the vanishing viscosity one do not coincide are provided in \cite{AttFla}.
\end{itemize}

We believe our main result holds on a wider class of domains and not only on the torus, but there are several technical issues which prevent a straightforward generalization and solving them is currently an open problem. Indeed, if the domain is a bounded open subset of $\mathbb{R}^d$, then a boundary condition must also be imposed and handled in the limit; in this regard, let us mention the recent work \cite{Hai}, in which it is shown that in certain scaling regimes (however different from our case) also the boundary condition must undergo a renormalization. If the setting is instead a compact manifold without boundary, the main challenge becomes finding examples of vector fields $\sigma_k$ for which the It\^o-Stratonovich correctors, as well as their limit once properly renormalized, can be computed explicitly. On the torus this task is greatly simplified by the presence of many simmetries, as it is shown in Section \ref{subsection 2.3 - SPDE in Ito form and definition of energy solutions}.

Let us highlight that even if in the discussion we have adopted a perturbative approach, motivating \eqref{sec 1 - STLE in compact Stratonovich form} as a stochastic variation of an originally deterministic problem, the equation is of interested by itself even when $b=0$ as it is related to the theory of passive scalars and the celebrated Kraichnan model of turbulence, see \cite{Cha}, \cite{Fal}. From the mathematical point of view, it has been treated in a very complete but rather technical way in \cite{LeJ}, \cite{LeJ2}; in Section \ref{subsection 4.2 - proof of strong uniqueness in the case b=0} we present a simple proof in the case $b=0$ of pathwise uniqueness of $L^2(\mathbb{T}^d)$-valued solutions under very mild assumptions on the noise (basically all isotropic divergence free noises for which the equation is well defined are included). To the best of our knowledge this result is new, since even in \cite{LeJ2} is suggested but not explicitly stated whether pathwise uniqueness can be proved, see the beginning of Section \ref{subsection 4.2 - proof of strong uniqueness in the case b=0} for more details.

Finally, let us mention the strong similarity between our technique and the one considered in \cite{FlaLuo2}.

%
%

\begin{plan}
The paper is structured as follows: in Section \ref{section 2 - preliminaries}, we introduce our notations and basic definitions; in Section \ref{section 3 - proof of the main result} we give a more precise statement and the proof of the main result. In Section \ref{section 4 - discussion of existence and uniqueness}, in order for the main result to be non vacuous, we give a proof of existence of energy solutions and we discuss the problem of their uniqueness. Finally in Section \ref{section 5 - the case d=1} we treat the case $d=1$, in which we show that we are not able to obtain an equivalent of the main result; still, from the modelling viewpoint, some interesting conclusions can be drawn.
\end{plan}

\section{Preliminaries}\label{section 2 - preliminaries}

%
%

In this section we provide all the notions necessary to give a meaning to \eqref{sec 1 - STLE in compact Stratonovich form} and its solutions; with this set up we will be able to prove the main result in the next section.
\subsection{Notations and functional setting}\label{subsection 2.1 - notation and functional setting}

%
%

We work on the $d$-dimensional torus, $\mathbb{T}^d=\mathbb{R}^d/(2\pi\mathbb{Z}^d)$, with periodic boundary condition. We denote by $L^2(\mathbb{T}^d;\mathbb{C})$ the set of complex-valued, square integrable function defined on $\mathbb{T}^d$, which is a Hilbert space endowed with the (normalized) inner product
\begin{equation}\nonumber
\langle f, g\rangle_{L^2} = \frac{1}{(2\pi)^d}\int_{\mathbb{T}^d} f(x)\overline{g}(x)\diff x,
\end{equation}
where $\overline{z}$ denotes the complex conjugate of $z$, $\vert z\vert^2 = z\,\overline{z}$; we denote by $\vert\cdot\vert_{L^2}$ the norm induced by $\langle\cdot,\cdot\rangle_{L^2}$. Under this inner product, $\{e_k\}_{k\in\mathbb{Z}^d}$ given by $e_k(x)= e^{i\, k\cdot x}$ is a complete orthonormal system ($k\cdot x= \sum_{i=1}^d k_i x_i$ denoting the standard inner product in $\mathbb{R}^d$). Any element $f\in L^2(\mathbb{T}^d;\mathbb{C})$ can be written uniquely in Fourier series as
\begin{equation}\nonumber
f = \sum_{k\in\mathbb{Z}^d} f_k\, e_k, \quad f_k = \langle f, e_k\rangle_{L^2} ,
\end{equation}
where the series is convergent in $L^2(\mathbb{T}^d;\mathbb{C})$ and it satisfies
\begin{equation}\nonumber
\vert f\vert_{L^2}^2 = \sum_{k\in\mathbb{Z}^d} \vert f_k\vert^2.
\end{equation}
An element $f$ is real-valued if and only if $f_{-k}=\overline{f}_k$ for every $k\in\mathbb{Z}^d$. We denote the set of square integrable, real-valued functions by $L^2(\mathbb{T}^d)=L^2$. The formulas above hold more generally for $f\in L^2(\mathbb{T}^d;\mathbb{R}^d)$ if we interpret $f_k$ as the $\mathbb{C}^d$-valued vector with components $f_k^{(j)} = \langle f^{(j)}, e_k\rangle_{L^2}$.

We will always deal with real-valued functions, but for the sake of calculations it is more convenient to use complex Fourier series; for the same reason we work on $\mathbb{T}^d$ defined as above rather than $\mathbb{R}^d/\mathbb{Z}^d$. We stress however that the results are independent of this choice and can be obtained in the same way by using real Fourier series or $\mathbb{R}^d/\mathbb{Z}^d$.

%
%

We consider the Sobolev spaces $H^\alpha(\mathbb{T}^d)$, $\alpha\in\mathbb{R}$, given by
\begin{equation}\nonumber
H^\alpha(\mathbb{T}^d) = \Big\{f=\sum_k f_k\,e_k\, \Big\vert\, f_{-k}=\overline{f}_k,\, \sum_k \big(1+\vert k\vert^2\big)^\alpha \vert f_k\vert^2 <\infty\Big\},
\end{equation}
see \cite{Tem} for more details. Then the space of test functions $C^\infty(\mathbb{T}^d)$ corresponds to $\cap_\alpha H^\alpha(\mathbb{T}^d)$ and its dual $C^\infty(\mathbb{T}^d)'$, the space of distributions, to $\cup_\alpha H^\alpha(\mathbb{T}^d)$. We denote by $\langle\cdot,\cdot\rangle$ also the duality pairing between them.

Given $f\in L^2(\mathbb{T}^d;\mathbb{R}^d)$, we say that $f$ is divergence free in the sense of distributions if
\begin{equation}\nonumber
\langle f, \nabla\varphi\rangle = 0 \qquad \forall\, \varphi\in C^\infty(\mathbb{T}^d).
\end{equation}
It's easy to check that $f$ is divergence free if and only if $f_k\cdot k = 0$ for all $k\in\mathbb{Z}^d$. Consider the subspace
\begin{equation}\nonumber
\mathcal{H}= \bigg\{f\in L^2(\mathbb{T}^d;\mathbb{R}^d) \text{ such that } \int_{\mathbb{T}^d} f =0 \text{ and } f \text{ is divergence free}\bigg\}.
\end{equation}
$\mathcal{H}$ is a closed linear subspace of $L^2(\mathbb{T}^d;\mathbb{R}^d)$ and so the orthogonal projection $\Pi:L^2(\mathbb{T}^d;\mathbb{R}^d)\to \mathcal{H}$ is a linear continuous operator. $\Pi$ can be represented in Fourier series by
\begin{equation}\nonumber
\Pi: f=\sum_{k\in\mathbb{Z}^d} f_k\, e_k \mapsto \Pi f = \sum_{k\in\mathbb{Z}^d} P_k f_k\, e_k,
\end{equation}
where $P_k\in \mathbb{R}^{d\times d}$ is the d-dimensional projection on $k^\perp$, $P_k = I - \frac{k}{\vert k\vert}\otimes \frac{k}{\vert k\vert}$, whenever $k\neq 0$ and we set $P_0\equiv 0$. $\Pi$ can be extended to a continuous linear operator from $H^\alpha(\mathbb{T}^d;\mathbb{R}^d)$ to itself for any $\alpha\in\mathbb{R}$. We also define the projectors $\Pi_N$ on the space of Fourier polynomials of degree at most $N$ by
\begin{equation}\nonumber
f=\sum_k f_k\,e_k\mapsto \Pi_Nf = \sum_{k: \vert k\vert\leq N} f_k\, e_k,
\end{equation}
where $\Pi_N:C^\infty(\mathbb{T}^d)'\to C^\infty(\mathbb{T}^d)$.

\subsection{Construction of the noise $W(t,x)$}\label{subsection 2.2 - contruction of the noise}

%
%

We have introduced the space $\mathcal{H}$ and the projector $\Pi$ because we want to deal with an $\mathcal{H}$-valued noise $W$; the reason for this choice will become clear in Section \ref{subsection 2.3 - SPDE in Ito form and definition of energy solutions}. We are first going to construct $W$ by giving an explicit Fourier representation, but then we will also provide a more elegant, abstract construction.

Set $\mathbb{Z}^d_0=\mathbb{Z}^d\setminus \{0\}$ and consider $\Lambda\subset\mathbb{Z}^d$ such that $\Lambda$ and $-\Lambda$ form a partition of $\mathbb{Z}^d_0$. Consider a collection
\begin{equation}\nonumber
\Big\{ B^{(j)}_k, k\in\mathbb{Z}^d_0, 1\leq j\leq d-1\Big\}
\end{equation}
of standard, real valued, independent $\mathcal{F}_t$-Brownian motions, defined on a filtered probability space $(\Omega,\mathcal{F},\mathcal{F}_t,\mathbb{P})$, $\{\mathcal{F}_t\}_{t\geq 0}$ being a normal filtration (see \cite{Rev}). Define
\begin{equation}\nonumber
W^{(j)}_k := \begin{cases} B^{(j)}_k + iB^{(j)}_{-k}\qquad & \text{if }\ k\in\Lambda\\
B^{(j)}_k - iB^{(j)}_{-k} & \text{if }\ k\in-\Lambda \end{cases}.
\end{equation}
In this way, $\{W^{(j)}_k, k\in\mathbb{Z}^d_0, 1\leq j\leq d-1\}$ is a collection of standard complex valued Brownian motions (namely complex processes with real and complex part given by independent real BM) such that $W^{(j)}_{-k}=\overline{W}^{(j)}_k$ and $W^{(j)}_k, W^{(m)}_l$ are independent whenever $k\neq \pm l$, $j\neq m$. We denote by $[M,N]$ the quadratic covariation process, which is defined for any couple $M$, $N$ of square integrable real semimartingales (see for instance \cite{Rev}) and we extend it by bilinearity to the analogue complex valued processes.
Observe that by result of bilinearity it holds
\begin{equation}\nonumber
\Big[W^{(j)}_k,W^{(j)}_k\Big]_t = 0,\quad \Big[W^{(j)}_k, W^{(j)}_{-k}\Big]_t = 2t,
\end{equation}
and therefore
\begin{equation}\nonumber
\Big[W^{(j)}_k,W^{(m)}_l\Big]_t = 2t\,\delta_{j,m}\,\delta_{k,-l} = 2t\,\delta_{j-m}\,\delta_{k+l}.
\end{equation}
We omit the details, but it's easy to check all the stochastic calculus rules, in particular It\^o formula and It\^o isometry, can be extended by bilinearity to the case of complex valued semimartingales.

For any $k\in\Lambda$, let $\{a_k^{(1)},\ldots, a_k^{(d-1)}\}$ be an orthonormal basis of $k^\perp$. Then $\{k/\vert k\vert, a_k^{(1)},\ldots, a_k^{(d-1)}\}$ form an orthonormal basis of $\mathbb{R}^d$ and it holds
\begin{equation}\nonumber
P_k = a_k^{(1)}\otimes a_k^{(1)} + \ldots +a_k^{(d-1)}\otimes a_k^{(d-1)} \quad \forall\, k\in\Lambda;
\end{equation}
for $k\in -\Lambda$ we can set $a_k^{(j)} = a_{-k}^{(j)}$ and the above identity still holds.

Let $\{\theta_k, k\in\mathbb{Z}^d_0\}$ be a collection of real constants such that $\theta_k=\theta_{-k}$ and satisfying suitable conditions, which will be specified later. We set
\begin{equation}\label{sec 2.2 - definition of the noise W(t,x)}
W(t,x):= \sum_{k\in\mathbb{Z}^d_0} \theta_k \Bigg(\sum_{j=1}^{d-1} a_k^{(j)}\,W_k^{(j)}(t)\Bigg) e_k(x) .
\end{equation}
From now on, whenever it doesn't create confusion, we will only write the indices $k,j$ without specifying their index sets, in order for the notation not to become too burdensome. Observe that, for fixed $t$, $W(t,\cdot)$ is already written in its Fourier decomposition and by the definitions of $a_k^{(j)}$ and $W_k(j)$ it's a real, mean zero, divergence free random distribution. It only remains to show that, for fixed $t$, $W(t,\cdot)$ belongs $\mathbb{P}$-a.s. to $L^2(\mathbb{T}^d;\mathbb{R}^d)$. Indeed, denoting by $\mathbb{E}$ the expectation with respect to $\mathbb{P}$, we have
\begin{equation}\nonumber
\mathbb{E}\Big[\,\vert W(t,\cdot)\vert_{L^2}^2\Big]
= \mathbb{E}\Bigg[ \sum_k  \theta_k^2\, \Big\vert \sum_j a_k^{(j)}W_k^{(j)}(t)\Big\vert^2\Bigg]
= 2t(d-1)\sum_k \theta_k^2
\end{equation}
and therefore, under the conditions
\begin{equation}\label{sec 2.2 - condition on the coefficients}\tag{H1}
\theta_{-k}=\theta_k\ \forall\, k,\quad \sum_k \theta_k^2<\infty,
\end{equation}
$W(t,\cdot)$ is a well defined random variable belonging to $L^2(\Omega,\mathcal{F},\mathbb{P};\mathcal{H})$. From the point of view of mathematical rigour, we should have first done the above calculation when summing over finite $k$ and then shown that, under condition \eqref{sec 2.2 - condition on the coefficients}, the sequence of finite sums is Cauchy; this can be easily checked and we omit it for the sake of simplicity. With similar calculations, exploiting Gaussianity and Kolmogorov continuity criterion, it can be shown that, as an $\mathcal{H}$-valued process, up to modification $W$ has paths in $C^{\alpha}([0,T];\mathcal{H})$ for any $\alpha<1/2$, see for instance \cite{DaP}.

%
%

We now show an alternative, more abstract construction of $W$. Let $\theta_k$ be some real coefficients satisfying \eqref{sec 2.2 - condition on the coefficients} as before and define the Fourier multiplier
\begin{equation}\nonumber
\Theta: f=\sum_k f_k\, e_k\mapsto \Theta f= \sum_k \theta_k f_k\, e_k.
\end{equation}
Then $\Theta$ is a continuous, self-adjoint operator from $H^\alpha(\mathbb{T}^d;\mathbb{R}^d)$ to itself which commutes with $\Pi$; condition \eqref{sec 2.2 - condition on the coefficients} implies that $\Theta$ is an Hilbert-Schmidt operator, namely $\Theta^\ast \Theta=\Theta^2$ is a trace class operator:
\begin{equation}\nonumber
\Theta^2: f=\sum_k f_k\, e_k\mapsto \Theta^2 f= \sum_k \theta_k^2\, f_k\, e_k.
\end{equation}
Now let $\widetilde{W}$ be a cylindrical Wiener process on $L^2(\mathbb{T}^d;\mathbb{R}^d)$ (in the sense of \cite{DaP}): a Gaussian distribution valued process with covariance
\begin{equation}\nonumber
\mathbb{E}\big[\langle \widetilde{W}_t,\varphi\rangle\,\langle \widetilde{W}_s,\psi\rangle \big]
= (t\wedge s)\langle \varphi,\psi\rangle
\quad \forall\, \varphi,\psi\in C^\infty(\mathbb{T}).
\end{equation}
Then it can be shown that, up to modifications, $\widetilde{W}$ has paths in $C^\alpha([0,T],H^\beta)$ for any $\alpha<1/2$ and for any $\beta<-d/2$. If we define
\begin{equation}\nonumber
W := \Theta\Pi \widetilde{W},
\end{equation}
then $W$ is a Wiener process on $\mathcal{H}$. This construction is useful as it shows that we can consider, on a given filtered probability space with a given noise $\widetilde{W}$, several different $W$ just by varying the deterministic operator $\Theta$. By construction, $W$ has covariance given by
\begin{equation}\nonumber\begin{split}
\mathbb{E}\big[\langle W_t,\varphi\rangle\,\langle W_s,\psi\rangle \big]
& = (t\wedge s)\langle (\Theta\Pi)^\ast\varphi,(\Theta\Pi)^\ast\psi\rangle\\
& = (t\wedge s)\langle \Theta^2\Pi\varphi,\psi\rangle
\qquad \quad \quad \forall\, \varphi,\psi\in C^\infty(\mathbb{T}).
\end{split}\end{equation}
%
%
It can be checked that $W$ defined as above is space homogeneous, namely its distribution is invariant under space translations $W(t.\cdot)\mapsto W(t,x+\cdot)$ for any $x\in\mathbb{T}^d$. This is a consequence of the fact that $\widetilde{W}$ is space homogeneous and $W$ is defined by a Fourier multiplier. For our purposes, we want it to be isotropic as well.

Let $E_\mathbb{Z}(d)$ denote the group of linear isometries of $\mathbb{R}^d$ into itself which leave $\mathbb{Z}^d$ invariant; it is the group generated by swaps
\begin{equation}\nonumber
(x_1,\ldots,x_i,\ldots,x_j,\ldots, x_d)\mapsto (x_1,\ldots,x_j,\ldots,x_i,\ldots, x_d)
\end{equation}
and reflections
\begin{equation}\nonumber
(x_1\ldots x_{i-1},x_i,x_{i+1}\ldots x_d)\mapsto (x_1\ldots x_{i-1},-x_i,x_{i+1}\ldots x_d).
\end{equation}
To see this, observe that if $O\in E_\mathbb{Z}(d)$, then for any element $e_i$ of the canonical basis it holds $Oe_i\in\mathbb{Z}^d$ and $\vert Oe_i\vert=1$, which necessarily implies that $Oe_i=\pm e_j$ for another index $j$.
$W$ is isotropic if its law is invariant under transformations $W(t,\cdot)\mapsto W(t,O\cdot)$ for all $O\in E_\mathbb{Z}(d)$ . In order to have an isotropic noise, we impose the following condition on the coefficients $\theta_k$:
\begin{equation}\label{sec 2.2 - isotropy condition}\tag{H2}
\theta_k = \theta_{Ok}\quad\, \forall\, O\in E_{\mathbb{Z}}(d).
\end{equation}
Tipical choices of $\theta_k$ will be of the form $\theta_k = F(\vert k\vert)$, where $F:\mathbb{R}_{\geq 0}\to\mathbb{R}_{\geq 0}$ is a function with sufficient decay at infinity; for instance we can take $F$ with finite support, or $F(r)=r^{-\alpha}$ for some $\alpha>d/2$. However all statements in the next section hold in general as long as \eqref{sec 2.2 - condition on the coefficients} and \eqref{sec 2.2 - isotropy condition} are satisfied.

\subsection{STLE in It\^o form and definition of energy solutions}\label{subsection 2.3 - SPDE in Ito form and definition of energy solutions}

We can now write explicitly \eqref{sec 1 - STLE in compact Stratonovich form} and find the corresponding It\^o formulation.
In order to simplify the exposition, we will do all the computations as if we were summing over a finite number of $k$ and find the right conditions under which every sum is well defined. Rigorously speaking, we should use an approximation argument and check that the finite series form a Cauchy sequence, but we skip this technical part, which can be easily verified.

%
%

Let $W$ be given as in \eqref{sec 2.2 - definition of the noise W(t,x)}, then \eqref{sec 1 - STLE in compact Stratonovich form} can be formulated as
\begin{equation}\label{sec 2.3 - equation in Stratonovich form}
\diff u = b\cdot\nabla u\diff t + \sum_{j,k} \theta_k\, e_k\, a_k^{(j)}\cdot \nabla u\circ \diff W^{(j)}_k.
\end{equation}
Equation \eqref{sec 2.3 - equation in Stratonovich form} must be interpreted in integral form: a process $u$ is a strong (from the analytical point of view) solution if $\mathbb{P}$-a.s. the following identity is satisfied for every $t,x$ (and $u$ is pogressively measurable and sufficiently regular for it to be meaningful):
\begin{equation}\nonumber\begin{split}
u(t,x)-u(0,x) = & \int_0^t b(s,x)\cdot\nabla u(s,x)\diff s\\ &+ \sum_{j,k}\theta_k \int_0^t e_k(x)\,a^{(j)}_k\cdot\nabla u(s,x)\circ \diff W_k^{(j)}(s).
\end{split}\end{equation}
Since in general Stratonovich integral is not so easy to control, we prefer to pass to the equivalent formulation in It\^o form:
\begin{equation}\nonumber\begin{split}
\diff u
& = b\cdot\nabla u\diff t + \sum_{j,k} \theta_k\, e_k\, a_k^{(j)}\cdot \nabla u \diff W^{(j)}_k + \frac{1}{2}\sum_{j,k} \theta_k\, e_k \diff \big[a_k^{(j)}\cdot\nabla u, W^{(j)}_k\big]\\
& = b\cdot\nabla u\diff t + \sum_{j,k} \theta_k\, e_k\, a_k^{(j)}\cdot \nabla u\diff W^{(j)}_k + \sum_{j,k}  \theta_k^2\, e_k\,a_k^{(j)}\cdot\nabla\Big( e_{-k}\,a_{-k}^{(j)}\cdot\nabla u\Big) \diff t\\
& = b\cdot\nabla u\diff t + \sum_{j,k} \theta_k\, e_k\, a_k^{(j)}\cdot \nabla u\diff W^{(j)}_k + \sum_{k,j} \theta_k^2\, \text{Tr}\Big(a_k^{(j)}\otimes a_k^{(j)}\, D^2u\Big)\diff t\\
& = b\cdot\nabla u\diff t + \sum_{j,k} \theta_k\, e_k\, a_k^{(j)}\cdot \nabla u\diff W^{(j)}_k + \text{Tr}\left( \Big( \sum_{k} \theta_k^2 P_k\Big) \, D^2u\right)\diff t.
\end{split}\end{equation}
In the above computation we exploited many of the properties of $a^{(j)}_k$ and $W^{(j)}_k$ highlighted in the previous section: $d[W^{(j)}_k,W^{(l)}_m] = 2\delta_{j,l}\delta_{k,-m}\,dt$, $a_k^{(j)}\cdot k=0$, $a_k^{(j)}=a^{(j)}_{-k}$. It remains to compute more explicitly the matrix appearing in the last line on the right hand side:
\begin{equation}\nonumber
\sum_k \theta_k^2 P_k
= \sum_k \theta_k^2 \left(I - \frac{k}{\vert k\vert}\otimes \frac{k}{\vert k\vert}\right)
= \left( \sum_k \theta_k^2\right) I - \sum_k \theta_k^2 \frac{k}{\vert k\vert}\otimes \frac{k}{\vert k\vert}.
\end{equation}
By the isotropy condition \eqref{sec 2.2 - isotropy condition} of $\theta_k$, whenever $i\neq j$, using the change of variables $k\mapsto \tilde{k}$ that switches the sign of the $i$-th component, we have
\begin{equation}\nonumber
\left( \sum_k \theta_k^2 \frac{k}{\vert k\vert}\otimes \frac{k}{\vert k\vert}\right)_{ij}
= \sum_k \theta_k^2\, \frac{k^{(i)}\,k^{(j)}}{\vert k\vert^2}
= \sum_{\tilde{k}} \theta_{\tilde{k}}^2\, \frac{\tilde{k}^{(i)}\,\tilde{k}^{(j)}}{\vert \tilde{k}\vert^2}
= \sum_k \theta_k^2\, \frac{-k^{(i)}\,k^{(j)}}{\vert k\vert^2} =0.
\end{equation}
Instead, when $i=j$, using a change of variables $k\mapsto \tilde{k}$ that swaps the $i$-th component with the $l$-th one, we obtain
\begin{equation}\nonumber
\left( \sum_k \theta_k^2 \frac{k}{\vert k\vert}\otimes \frac{k}{\vert k\vert}\right)_{ii}
= \sum_k \theta_k^2\, \frac{{k^{(i)}}^2}{\vert k\vert^2}
= \sum_k \theta_k^2\, \frac{{k^{(l)}}^2}{\vert k\vert^2}
= \left( \sum_k \theta_k^2 \frac{k}{\vert k\vert}\otimes \frac{k}{\vert k\vert}\right)_{ll}
\end{equation}
and therefore, summing over $i$,
\begin{equation}\nonumber
\left( \sum_k \theta_k^2 \frac{k}{\vert k\vert}\otimes \frac{k}{\vert k\vert}\right)_{ii}
= \frac{1}{d} \sum_k \theta_k^2 \frac{\vert k\vert^2}{\vert k\vert^2}
= \frac{1}{d} \sum_k \theta_k^2.
\end{equation}
In conclusion, we have obtained
\begin{equation}\label{sec 2.3 - definition of the constant c}
\sum_k \theta_k^2\, P_k = \frac{d-1}{d} \Bigg(\sum_k \theta_k^2\Bigg) I=: c\,I,
\end{equation}
so that equation \eqref{sec 2.3 - equation in Stratonovich form} has corresponding It\^o formulation
\begin{equation}\label{sec 2.3 - equation in Ito form}
\diff u = b\cdot\nabla u\diff t + c\,\Delta u\diff t + \sum_{j,k} \theta_k\, e_k\, a_k^{(j)}\cdot \nabla u\diff W^{(j)}_k.
\end{equation}
We have actually only shown that formally \eqref{sec 2.3 - equation in Stratonovich form} implies \eqref{sec 2.3 - equation in Ito form}, but the same calculations done backward show that the two formulations are equivalent, whenever $u$ is a smooth solution. Observe that condition \eqref{sec 2.2 - condition on the coefficients} on the coefficients $\theta_k$ is necessary in order to give a meaning to equation \eqref{sec 2.3 - equation in Stratonovich form}: otherwise, passing to the It\^o formulation we would find a term of the form \textquotedblleft $+\infty\,\Delta u$\textquotedblright\  which is ill-defined even in the case $u$ had very good regularity.
Let us stress that, even if writing the It\^o formulation we find a diffusion term, this is actually a \textquotedblleft fake Laplacian\textquotedblright : the nature of the equation is still hyperbolic and it can be solved by characteristics; moreover, in the case $\text{div}\, b=0$, it can be checked that the energy $\vert u\vert_{L^2}$ is (formally) invariant, while in a real diffusion it would be dissipated.
We have done the computations leading to \eqref{sec 2.3 - definition of the constant c} explicitly but we could have also derived it by the following reasoning: the matrix
\begin{equation}\nonumber
A= \sum_k \theta_k^2\, P_k
\end{equation}
is a symmetric and semipositive definite; by the isotropy condition it follows that $O^T A O = A$ for all $O\in E_{\mathbb{Z}}(d)$ and therefore necessarily $A=c I$ for some constant $c$. Indeed, if $v$ is an eigenvector for $A$, by the isotropy condition so is $Ov$, with respect to the same eigenvalue, for all $O\in E_{\mathbb{Z}}(d)$; this immediately implies that the associated eigenspace is the whole $\mathbb{R}^d$. But then taking the trace on both sides we find
\begin{equation}\nonumber
(d-1)\sum_k\theta_k^2 = dc,
\end{equation}
which gives \eqref{sec 2.3 - definition of the constant c}. This shows that the presence of $\Delta$ is strictly related to isotropy of the noise.

%
%
Since we are interested in studying weak (in the analytical sense) solutions of equation \eqref{sec 2.3 - equation in Ito form}, we need to rewrite it in a suitable way by testing against test functions in $C^\infty(\mathbb{T}^d)$. Recalling that, for any $k$ and $j$, $x\mapsto a_k^{(j)}\,e_k(x)$ is divergence free by construction, the weak formulation then corresponds to:
\begin{equation}\label{sec 2.3 - equation in Ito form, weak formulation in differential form}\begin{split}
\diff\langle u,\varphi\rangle = & -\langle u, \text{div}(b\varphi)\rangle\diff t + c\,\langle u,\Delta \varphi\rangle\diff t\\ & - \sum_{j,k} \theta_k\, \langle u, e_k\, a_k^{(j)}\cdot \nabla \varphi\rangle\diff W^{(j)}_k\quad \forall\,\varphi\in C(\mathbb{T}^d),
\end{split}\end{equation}
where as usual the above equation must be interpreted in the integral sense. In order for the term $\langle u, \text{div}(b\varphi)\rangle = \langle u, \text{div}b\,\varphi + b\nabla\varphi)\rangle$ to be well defined, we need at least to require the following assumption on $b$:
\begin{equation}\tag{A1}\label{sec 2.3 - assumption 1 on b}
b\in L^2(0,T;L^2),\quad \text{div}\,b\in L^2(0,T;L^2).
\end{equation}
It is natural in the definition of weak solution to require weak continuity in time of the solution. We denote by $C([0,T];L^2_w)$ the space of functions $f:[0,T]\to L^2$ which are continuous w.r.t. the weak topology of $L^2$, namely $f(s)\rightharpoonup f(t)$ as $s\to t$. For more details on the weak topology, we refer to \cite{Bre}. We are now ready to give with the following definition.
\begin{definition}\label{definition sec 2.3 - weak solution}
Let $(\Omega,\mathcal{F},\mathcal{F}_t,\mathbb{P})$ be a filtered probability space, with normal filtration $\{\mathcal{F}_t\}$, on which a collection $\{B_k^{(j)}, k\in\mathbb{Z}^d_0, 1\leq j\leq d-1\}$ of independent, standard $\mathcal{F}_t$-Brownian motions is defined. Let $W$ be defined as in \eqref{sec 2.2 - definition of the noise W(t,x)}, for given coefficients $\{\theta_k\}_k$ satisfying \eqref{sec 2.2 - condition on the coefficients}, \eqref{sec 2.2 - isotropy condition}. We say that an $\mathcal{F}_t$-progressively measurable, $L^2$-valued process $u$, with paths in $C([0,T];L^2_w)$, satisfying
\begin{equation}\label{sec 2.3 - condition of bounded energy in the definition of weak solution}
\int_0^T \mathbb{E}\big[\vert u(t)\vert^2_{L^2}\big]\diff t <\infty,
\end{equation}
is a \textbf{weak solution} (in the analytical sense) on the interval $[0,T]$ of the equation
\begin{equation}\label{sec 2.3 - equation in Stratonovich compact form, useful repetition}
\diff u = b\cdot\nabla u\diff t + \circ \diff W\cdot\nabla u
\end{equation}
if, for every $\varphi\in C^\infty(\mathbb{T}^d)$, $\mathbb{P}$-a.s. the following identity holds for all $t\in [0,T]$:
\begin{equation}\label{sec 2.3 - equation in Ito form, weak formulation in integral form}\begin{split}
\langle u(t),\varphi\rangle - \langle u(0),\varphi\rangle
=& -\int_0^t \langle u(s), \text{div}(b\varphi)\rangle\diff s + c\int_0^t \langle u(s),\Delta \varphi\rangle\diff s\\
& - \sum_{j,k} \theta_k\, \int_0^t \langle u(s), e_k\, a_k^{(j)}\cdot \nabla \varphi\rangle\diff W^{(j)}_k(s).
\end{split}\end{equation}
\end{definition}

In order to show that it is a good definition, we need to prove that equation \eqref{sec 2.3 - equation in Ito form, weak formulation in integral form} is meaningful. By assumption \eqref{sec 2.3 - assumption 1 on b} and condition \eqref{sec 2.3 - condition of bounded energy in the definition of weak solution}, it holds
\begin{equation}\nonumber\begin{split}
\mathbb{E}&\bigg[\bigg\vert \int_0^t \langle u(s), \text{div}(b\varphi)\rangle\diff s \bigg\vert\bigg]\\
&\leq \Vert \varphi\Vert_{W^{1,\infty}} 
\big(\Vert b\Vert_{L^2(0,T;L^2)} + \Vert \text{div}b\Vert_{L^2(0,T;L^2)}\big)
\, \sqrt{\int_0^T \mathbb{E}\big[\vert u(t)\vert^2_{L^2}\big]\diff t} <\infty.
\end{split}\end{equation}
Since $u$ is $\mathcal{F}_t$-progressively measurable, the real-valued process $t\mapsto \langle u(s), e_k\, a_k^{(j)}\cdot \nabla \varphi\rangle$ is also $\mathcal{F}_t$-progressively measurable and can be integrated with respect to $W_k^{(j)}$, for any $k$ and $j$. Therefore we only need to check that the infinite series is convergent, in a suitable sense. By It\^o isometry we have
\begin{equation}\nonumber\begin{split}
\mathbb{E}\Bigg[\Big\vert \sum_{j,k} \theta_k & \int_0^t \langle u(s), e_k\, a_k^{(j)}\cdot \nabla \varphi\rangle\diff W^{(j)}_k(s) \Big\vert^2 \Bigg]\\
& = 2 \sum_{j,k}\theta_k^2\, \mathbb{E}\Bigg[\int_0^t\vert \langle u(s),e_k\, a_k^{(j)}\cdot \nabla \varphi\rangle\vert ^2\diff s\Bigg]\\
& \leq 2\, \sup_k \theta_k^2\, \sum_{j,k} \mathbb{E}\Bigg[\int_0^T \vert\langle u(s)\,\nabla\varphi, a_k^{(j)}\,e_k\rangle\vert^2\diff s \Bigg]\\
& \leq 2\, \sup_k \theta_k^2\, \int_0^T\mathbb{E}\big[\vert u(s)\,\nabla\varphi\vert_{L^2}^2\big]\diff s\\
& \leq 2\, \sup_k \theta_k^2\, \Vert\nabla\varphi\Vert_{\infty}^2  \int_0^T\mathbb{E}\big[\vert u(s)\vert_{L^2}^2\big]\diff s
\end{split}\end{equation}
and the last term is finite since $u$ satisfies \eqref{sec 2.3 - condition of bounded energy in the definition of weak solution} and $\theta_k$ satisfy \eqref{sec 2.2 - condition on the coefficients}. In the above calculations we have exploited the fact that $\{a_k^{(j)}\, e_k, k\in\mathbb{Z}^d_0, 1\leq j\leq d-1\}$ is an (incomplete) orthonormal system in $L^2(\mathbb{T}^d;\mathbb{C}^d)$.

%
%
Let us now briefly discuss the energy balance for equation \eqref{sec 2.3 - equation in Stratonovich compact form, useful repetition}. If $u$ were a classical smooth solution of the deterministic linear transport equation
\begin{equation}\nonumber
\partial_t u = b\cdot \nabla u + v\cdot\nabla u,
\end{equation}
with $b$ as before and $v=v(t,x)$ being a divergence free vector field (both with periodic boundary condition), then we would have
\begin{equation}\nonumber\begin{split}
\frac{\diff}{\diff t}\int_{\mathbb{T}^d} u^2(t,x)\diff x
& = \int_{\mathbb{T}^d} 2\,u(t,x)(b(t,x)+v(t,x))\cdot \nabla u(t,x)\diff x\\
& = \int_{\mathbb{T}^d} (b(t,x)+v(t,x))\cdot \nabla (u^2)(t,x)\diff x\\
& = -\int_{\mathbb{T}^d} (\text{div} b)(t,x)\,u^2(t,x)\diff x\\
& \leq \Vert \text{div} b(t)\Vert_\infty \int_{\mathbb{T}^d} u^2(t,x)\diff x,
\end{split}\end{equation}
and therefore by Gronwall's lemma we would obtain
\begin{equation}\label{sec 2.3 - energy inequality, preliminary version}
\vert u(t)\vert_{L^2}^2 \leq \vert u(0)\vert_{L^2}^2\,\exp\Big\{\int_0^t \Vert \text{div} b(s)\Vert_\infty\diff s \Big\}.
\end{equation}
It is therefore natural to impose the following condition on $b$:
\begin{equation}\tag{A2}\label{sec 2.3 - assumption 2 on b}
\text{div}\,b \in L^1(0,T;L^\infty).
\end{equation}
Using the properties of Stratonovich integral (or if one prefers using a Wong-Zakai approximation technique), it can be shown that, whenever $u$ is a smooth solution of (STLE), the above calculation still holds, since by construction $W(t,\cdot)$ is divergence free. However, since equation \eqref{sec 2.3 - equation in Stratonovich compact form, useful repetition} is hyperbolic in nature, we don't expect solutions with initial data only in $L^2$ to become more regular and in this case the above reasoning doesn't hold. By approximation with smooth solutions, we can still at least expect the final inequality \eqref{sec 2.3 - energy inequality, preliminary version} to hold also for weak solutions.

%
%

The above observations lead to the following notion of energy solutions for the Cauchy problem given by \eqref{sec 2.3 - equation in Stratonovich compact form, useful repetition} and an initial condition $u_0$:
\begin{definition}\label{definition sec 2.3 - energy solution} Given a deterministic initial condition $u_0\in L^2$, we say that $u$ is an \textbf{energy solution} of the Cauchy problem
\begin{equation}\nonumber
\begin{cases} \diff u = b\cdot\nabla u\diff t + \circ \diff W\cdot\nabla u\\ u(0)=u_0
\end{cases}
\end{equation}
if $u$ is a weak solution of \eqref{sec 2.3 - equation in Stratonovich compact form, useful repetition}, equation \eqref{sec 2.3 - equation in Ito form, weak formulation in integral form} is satisfied with $u(0)=u_0$ and the following \textbf{energy inequality} holds:
\begin{equation}\label{sec 2.3 - energy inequality}
\sup_{t\in [0,T]} \Big\{e^{-1/2 \int_0^t \Vert \text{div} b(s,\cdot)\Vert_\infty\diff s}\,\vert u(t)\vert_{L^2} \Big\} \leq \vert u_0\vert_{L^2} \qquad \mathbb{P}\text{-a.s.}
\end{equation}
\end{definition}

Let us finally define what we mean by convergence in probability in abstract topological spaces. If $X_n$ is a sequence of random variables defined on the same probability space, with values in $(E,\tau,\mathcal{B}(\tau))$, where $\tau$ is a topology and $\mathcal{B}(\tau)$ is the associated Borel-$\sigma$ algebra, we say that $X_n\to X$ in probability if any subsequence of $\{X_n\}_n$ contains a subsequence which converges to $X$ $\mathbb{P}$-a.s.. We need this definition because we will work with convergence in probability in a non metrizable topology.

\section{Rigorous statement and proof of the main result}\label{section 3 - proof of the main result}
%
%
In this section we provide a rigorous statement of the main result and its proof. Throughout the section we consider a fixed, a priori given filtered probability space $(\Omega,\mathcal{F},\mathcal{F}_t,\mathcal{P})$ together with a collection $\{B_k^{(j)}, k\in\mathbb{Z}^d_0, 1\leq j\leq d-1\}$ of independent, standard $\mathcal{F}_t$-Brownian motions. However we consider different choices of the parameters $\{\theta_k\}_k$, so that we can obtain different space-time dependent noises $W(t,x)$ constructed from $\{B_k^{(j)}\}_{k,j}$ by \eqref{sec 2.2 - definition of the noise W(t,x)}. All these noises are still defined on the same probability space with respect to the same filtration; the drift $b$ is fixed. Whenever referring to energy solutions of \eqref{sec 2.3 - equation in Stratonovich compact form, useful repetition} we will therefore consider strong in the probabilistic sense solutions (i.e. progressively measurable w.r.t. $\mathcal{F}_t$) all defined on the same probability space. The main result can then be stated as follows.
%
%
\begin{theorem}\label{theorem sec 3 - main result}
Let $\{\theta_k^N, k\in\mathbb{Z}^d_0, N\in\mathbb{N} \}$ be a collection of real coefficients such that:
\begin{itemize}
\item[i)] For each $N$, $\{\theta_k^N\}_k$ satisfies \eqref{sec 2.2 - condition on the coefficients} and \eqref{sec 2.2 - isotropy condition}.
\item[ii)] It holds
\begin{equation}\label{sec 3 - condition on the coefficients in the statement of the main theorem}\tag{H3}
\lim_{N\to\infty} \frac{\sup_k (\theta_k^N)^2}{\sum_k (\theta_k^N)^2} = 0.
\end{equation}
\end{itemize}
Assume that $b$ satisfies \eqref{sec 2.3 - assumption 1 on b}, \eqref{sec 2.3 - assumption 2 on b} and the following:
\begin{itemize}
\item[(A3)] $b$  is such that, for any $\nu>0$, uniqueness holds in the class of weak $L^\infty(0,T;L^2)$ solutions of the parabolic Cauchy problem
\begin{equation}\label{sec 3 - parabolic cauchy limit problem in the statement of the main theorem}
\begin{cases} \partial_t u = \nu\Delta u + b\cdot\nabla u\\ u(0)=u_0 \end{cases}.
\end{equation}
\end{itemize}
Let $W^N$ denote the divergence free noises constructed from the coefficients $\{\theta^N_k\}_k$ as in \eqref{sec 2.2 - definition of the noise W(t,x)}. Then for any $\nu>0$ there exists a sequence of constants $\varepsilon^N$, which depend on the coefficients $\{\theta_k^N\}$, such that, for any $u_0\in L^2$, any sequence of energy solutions $u^N$ of the Cauchy problems
\begin{equation}\label{sec 3 - cauchy probelms for uN in the statement of the main theorem}
\begin{cases} \diff u^N= b\cdot\nabla u^N\diff t + \sqrt{\varepsilon^N} \circ \diff W^N\cdot\nabla u^N\\
u(0)=u_0
\end{cases}
\end{equation}
converge in probability, in $L^\infty(0,T;L^2)$ endowed with the weak-$\star$ topology, to the unique weak solution of the deterministic Cauchy problem \eqref{sec 3 - parabolic cauchy limit problem in the statement of the main theorem}.
In particular, the constants $\varepsilon^N$ can be taken as
\begin{equation}\label{sec 3 - choice of the varepsilonN in the statement of the main theorem}
\varepsilon^N = \nu\,\frac{d}{d-1}\Big(\sum_k(\theta_k^N)^2 \Big)^{-1}.
\end{equation}
\end{theorem}
%
%
\begin{proof}
%
%
The basic idea of the proof is the following: when we rewrite the transport SPDE in It\^o form, we can see that in general the It\^o-Stratonovich corrector term is well defined under more restrictive conditions than the It\^o integral. We can exploit this to our advantage by introducing a multiplicative renormalization $\sqrt{\varepsilon^N}$ under which the corrector term is uniformly bounded with respect to $N$, but then under condition \eqref{sec 3 - condition on the coefficients in the statement of the main theorem} the It\^o integrals become infinitesimal. To this aim, it is fundamental to have a uniform control on the energy of the solutions $u^N$ and that's why we work with energy solutions. We now formalize this reasoning properly.

%
%

By definition of energy solutions, we know that for each $N$ inequality \eqref{sec 2.3 - energy inequality} holds. In particular it follows that there exists a constant $K$, which only depends on $b$, such that
\begin{equation}\label{sec 3 - uniform energy estimate inside proof of the main theorem}
\sup_N \Vert u^N(\omega)\Vert_{L^\infty(0,T,L^2)}
= \sup_N \sup_{t\in [0,T]} \vert u^N(\omega,t)\vert_{L^2} \leq K\vert u_0\vert_{L^2} \quad \text{for }\mathbb{P}\text{-a.e. }\omega
\end{equation}
and
\begin{equation}\nonumber
\sup_N \int_0^T\mathbb{E}[\vert u^N(t)\vert_{L^2}^2]\diff t\leq K\vert u_0\vert_{L^2}.
\end{equation}
Rewriting the Cauchy problem in It\^o form, by the definition of energy solution we obtain that, for any $N$ and for any $\varphi\in C^{\infty}(\mathbb{T}^d)$, it holds
\begin{equation}\nonumber\begin{split}
\langle u^N(t),\varphi\rangle - \langle u_0,\varphi\rangle
=& -\int_0^t \langle u^N(s), \text{div}(b\varphi)\rangle\diff s + \varepsilon^N c^N \int_0^t \langle u^N(s),\Delta \varphi\rangle\diff s\\
& - \sqrt{\varepsilon^N}\,\sum_{j,k} \theta^N_k\, \int_0^t \langle u^N(s), e_k\, a_k^{(j)}\cdot \nabla \varphi\rangle\diff W^{(j)}_k(s),
\end{split}\end{equation}
where $c^N$ is defined as in \eqref{sec 2.3 - definition of the constant c}. With the choice \eqref{sec 3 - choice of the varepsilonN in the statement of the main theorem}, the equation becomes
\begin{equation}\label{sec 3 - identity for energy solutions uN inside proof main theorem}\begin{split}
\langle u^N(t),\varphi\rangle - \langle u_0,\varphi\rangle
=& -\int_0^t \langle u^N(s), \text{div}(b\varphi)\rangle\diff s + \nu \int_0^t \langle u^N(s),\Delta \varphi\rangle \diff s\\
& - \sqrt{\varepsilon^N}\,\sum_{j,k} \theta^N_k\, \int_0^t \langle u^N(s), e_k\, a_k^{(j)}\cdot \nabla \varphi\rangle\diff W^{(j)}_k(s).
\end{split}\end{equation}
Using estimates similar the ones of Section \ref{subsection 2.3 - SPDE in Ito form and definition of energy solutions}, it holds
\begin{equation}\nonumber\begin{split}
\varepsilon^N\, & \mathbb{E}\Bigg[ \Big(\sum_{j,k}\theta^N_k \int_0^T \langle u^N(s), e_k\, a_k^{(j)}\cdot \nabla \varphi\rangle\diff W^{(j)}_k(s)\Big)^2 \Bigg]\\
& \leq 2\varepsilon^N\, (\sup_k\theta_k^N)^2\,\Vert \nabla\varphi\Vert_{\infty}^2 \int_0^T\mathbb{E}[\vert u^N(t)\vert_{L^2}^2] \diff t\\
& \leq \widetilde{K} \Vert\nabla\varphi\Vert_{\infty}^2 \frac{(\sup_k\theta_k^N)^2}{\sum_k (\theta^N_k)^2}\to 0 \text{ as } N\to\infty
\end{split}\end{equation}
by assumption \eqref{sec 3 - condition on the coefficients in the statement of the main theorem}. Using the properties of It\^o integral and Doob's inequality, we deduce that for any fixed $\varphi$
\begin{equation}\nonumber
\sup_{t\in [0,T]} \bigg\vert\sum_{j,k}\theta^N_k\int_0^t \langle u^N(s), e_k\, a_k^{(j)}\cdot \nabla \varphi\rangle\diff W^{(j)}_k(s) \bigg\vert\to 0 \text{ in probability w.r.t. }\mathbb{P}.
\end{equation}
Let $\{\varphi_n\}_n$ be a countable dense subset of $C^\infty(\mathbb{T}^d)$; by a diagonal extraction argument, we can find a subsequence (which will still be denoted by $N$ for simplicity) and a set $\Gamma\subset \Omega$ with $\mathbb{P}(\Gamma)=1$ such that: the above process converges uniformly to 0 for every $\varphi_n$ and for every $\omega\in \Gamma$; inequality \eqref{sec 3 - uniform energy estimate inside proof of the main theorem} holds for every $\omega\in \Gamma$. Now let us consider a fixed $\omega\in \Gamma$ and the realizations $\{u^N(\omega)\}_N$. Since they are a bounded sequence $L^\infty(0,T;L^2)$, we can extract a subsequence (which depends on $\omega$) which is weak-$\star$ convergent to some $u\in L^\infty(0,T;L^2)$. Taking the limits on both sides of \eqref{sec 3 - identity for energy solutions uN inside proof main theorem}, since $\omega\in\Gamma$, we find that for every $n$
\begin{equation}\nonumber
\langle u(t),\varphi_n\rangle - \langle u_0,\varphi_n\rangle = -\int_0^t \langle u(s), \text{div}(b \varphi_n)\rangle\diff s + \nu \int_0^t \langle u(s),\Delta \varphi_n\rangle\diff s.
\end{equation}
By density we can extend the above equation for all $\varphi\in C^\infty(\mathbb{T}^d)$, so that $u$ is a weak solution of the Cauchy problem
\begin{equation}\nonumber\begin{cases}
\partial_t u = \nu\Delta u + b\cdot\nabla u\\ u(0)=u_0
\end{cases}.\end{equation}
By assumption (A3), the candidate limit is therefore unique; since the argument applies for any subsequence of $\{u^N(\omega)\}_N$, we conclude that the entire sequence is converging weakly-$\star$ to the unique solution of the above problem, without the need of selecting an $\omega$-dependent subsequence. Moreover the reasoning holds for any $\omega\in \Gamma$. Summarising, we have shown the existence of a subsequence of $\{u^N\}_N$ such that, for any $\omega\in\Gamma$, this subsequence converges in the weak-$\star$ topology of $L^\infty(0,T;L^2)$ to the unique solution of the above deterministic parabolic equation. Since the reasoning holds also for any subsequence of $\{u^N\}$, we conclude that convergence in probability in $L^\infty(0,T;L^2)$ endowed with weak-$\star$ topology holds.
\end{proof}

%
%

\begin{remark} Let us make some comments on the above result.

\begin{itemize}

\item[i)] For any $u^N$ solving \eqref{sec 3 - cauchy probelms for uN in the statement of the main theorem}, its expectation $\tilde{u}(t)=\mathbb{E}[u^N(t)]$ solves \eqref{sec 3 - parabolic cauchy limit problem in the statement of the main theorem}. Therefore the result can be expressed as the convergence in probability of $u^N$ to their mean value, which is a weak law of large numbers.

\item[ii)] Observe that whenever we consider coefficients $\{\theta^N_k\}$ satisfying \eqref{sec 3 - condition on the coefficients in the statement of the main theorem} such that $\sup_k \vert \theta^N_k\vert =1$ for all $N$ (some examples will be given shortly), the sequence of noises $W^N(t,x)$ has bounded norm in some distribution spaces, like $H^\alpha$ for $\alpha<-d/2$. Recalling that
\begin{equation}\nonumber
\mathbb{E}\big[\vert W^N(1,x)\vert_{L^2}^2\big]= 2(d-1)\sum_k(\theta_k^N)^2,
\end{equation}
we find that the constants $\varepsilon^N$ and $\nu$ must satisfy the relation
\begin{equation}\label{sec 3 - nu measures the irregularity/magnitude relation}
\nu = C(d)\lim_{N\to\infty}\varepsilon^N\,\mathbb{E}\big[\vert W^N(1,x)\vert_{L^2}^2\big],
\end{equation}
where $C(d)=2/d$ is a dimensional constant, independent of the probability space, the coefficients $\{\theta_k^N\}_{k,N}$ and the noises $W^N$ considered. Therefore the parameter $\nu$ appearing in the limit equation in front of the dissipation term $\Delta$ is measuring product of the spatial irregularity of the noise (in terms of its $L^2$ norm) and its magnitude.

\item[iii)] We illustrate some typical examples of coefficients $\theta_k^N$, widely used in other contexts, which satisfy \eqref{sec 2.2 - isotropy condition} and \eqref{sec 3 - condition on the coefficients in the statement of the main theorem}. Let $F:\mathbb{R}_{\geq 0}\to\mathbb{R}_{\geq 0}$ be a smooth, decreasing function of compact support with $F(0)=1$ and consider a sequence of positive real numbers $\alpha_N\to 0$; then we can take $\theta_k^N:= F(\alpha_N\vert k\vert)$. Other choices, for $\alpha_N$ infinitesimal, are
\begin{equation}\nonumber
\theta_k^N = (1+\alpha_N\vert k\vert^2)^\beta \text{ for some }\beta<-d/2,\quad \theta_k^N = (1+\vert k\vert^2)^{-d/2-\alpha_N}.
\end{equation}
We can also take $\theta^N_k=\mathbbm{1}_{B(0,1)}(\alpha_N\vert k\vert)$, where $\mathbbm{1}_A$ denotes the characteristic function of $A$. These examples can also be combined together to produce new ones. In terms of Fourier multipliers, some of the above examples are standard rescaled volume cutoffs in Fourier space, others correspond to operators like $(1-\alpha_N\Delta)^{-\beta}$ or $(1-\Delta)^{-d/2-\alpha_N}$.

\item[iv)] The theorem resembles a renormalization statement: different choices of the coefficients $\theta_k^N$, which can be spatial regularizations of a space-time white noise, require different multiplicative constants $\varepsilon^N$, but the final limit solves an equation which is independent of $\theta_k^N$, up to the choice of a 1-dimensional parameter $\nu$. We have however already pointed out in the introduction the presence of some degeneracy in our result. Indeed, different choices of the parameters $\theta_k^N$ can lead to very different limits for $W^N$ in terms of regularity: for instance, taking $(1+\alpha_N\vert k\vert)^{-\beta}$, the sequence will converge to white noise, while taking $(1+\vert k\vert^2)^{-1}\mathbbm{1}_{B(0,1)}(\alpha_N\vert k\vert)$ it will converge to a Gaussian free field (properly speaking, since we want divergence-free distributions, it will converge to the image under $\Pi$ of the aforementioned objects). However, in both cases the multiplicative constants $\varepsilon_N$ will still give convergence to the same limit. It is therefore unclear if the choice of such a renormalization is too strong, in the sense that it is ignoring too much information on the dynamics, and there is some more refined way to recover it, like an "higher order expansion" which not only measures the $L^2$-regularity of $W^N$ but also other norms.

\item[v)] We have required $b$ to satisfy \eqref{sec 2.3 - assumption 2 on b} in order to deal with energy solutions, but in principle the structure of the proof holds for any sequence $u^N$ of weak solutions satisfying a uniform bound of the form
\begin{equation}\nonumber
\sup_N \int_0^T \mathbb{E}[\vert u^N(s)\vert_{L^2}^2]\diff s \leq K
\end{equation}
up to paying the price of restricting ourselves to a weaker notion of convergence, namely weak convergence in $L^2(\diff \mathbb{P}\otimes \diff t; L^2)$. It's possible that more refined a priori estimates on the solutions $u^N$ provide this kind of bound under milder conditions on $b$ than \eqref{sec 2.3 - assumption 2 on b}.
\end{itemize}
\end{remark}

We now provide explicit sufficient conditions on $b$ under which assumption (A3) is satisfied. We give the statement in full generality, even when assumption \eqref{sec 2.3 - assumption 2 on b} does not hold.
\begin{lemma}\label{lemma sec 3 - sufficient condition for uniqueness of parabolic problem} Consider $b$ such that \eqref{sec 2.3 - assumption 1 on b} holds, as well as the following condition:
\begin{equation}\tag{A4}\label{sec 3 - KR type condition on b}
\begin{cases}
b\in L^{p_1}(0,T;L^{q_1}(\mathbb{T}^d))\quad
&\text{with } q_1\in (d,+\infty],\ p_1\in \big( \frac{2 q_1}{q_1 - d},+\infty\big]\\
\textnormal{div}\, b\in L^{p_2}(0,T;L^{q_2}(\mathbb{T}^d)) &\text{with } q_2\in \big(\frac{d}{2},+\infty\big],\ p_2\in \big( \frac{2 q_2}{2 q_2 - d},+\infty\big]
\end{cases}.\end{equation}
Then (A3) holds, i.e. we have uniqueness in the class of weak $L^\infty(0,T;L^2(\mathbb{T}^d))$ solutions of the Cauchy problem
\begin{equation}\label{sec 3 - parabolic Cauchy problem, useful repetition}\begin{cases}
\partial_t u = \nu\Delta u + b\cdot \nabla u\\
u(0)=u_0 \in L^2(\mathbb{T}^d)
\end{cases}.\end{equation}
\end{lemma}
\begin{proof}
Without loss of generality we can assume $\nu = 1$. By linearity, it suffices to show uniqueness for $u_0=0$. We first show that $u$ is also a mild solution of \eqref{sec 3 - parabolic Cauchy problem, useful repetition}. Indeed if $u$ is a weak solution of \eqref{sec 3 - parabolic Cauchy problem, useful repetition}, then for any interval $[s,t]\subset [0,T]$ and for any $\varphi\in C^\infty([s,t]\times\mathbb{T}^d)$ it holds
\begin{equation}\nonumber\begin{split}
\langle u(t),\varphi(t)\rangle - \langle u(s),\varphi (s)\rangle & = \int_s^t \langle u(r), (\partial_t+\Delta)\varphi(r)\rangle\diff r\\& - \int_s^t \langle u(r), \text{div}(b(r)\varphi(r))\rangle\diff r.
\end{split}\end{equation}
This can be accomplished by taking a partition $s=t_0<t_1<\ldots<t_n=t$ and applying the standard weak formulation on every interval $[t_i,t_{i+1}]$ with $\varphi(t_i)$, then summing over $i$ and then letting the mesh of the partition tend to 0. Recall that the heat kernel on the torus is given by
\begin{equation}\nonumber
P_t = \sum_k e^{-t\vert k\vert^2}\, e_k.
\end{equation}
Then testing $u$ on any interval $[\delta, t]$ with $\delta>0$ against the convolution with the heat kernel $P_{t-s}$ and letting $\delta\to 0$, using $u_0=0$ we obtain the mild formulation
\begin{equation}\nonumber
u(t) = \int_0^t P_{t-s} (\nabla (bu)-\text{div} b\, u)\diff s \quad \forall\, t\in [0,T].
\end{equation}
In order to conclude it suffices to show that the map
\begin{equation}\nonumber
u\mapsto \int_0^\cdot P_{\cdot-s} (\nabla (bu)-\text{div} b\, u)\diff s
\end{equation}
is a contraction of $L^\infty([0,T^\ast],L^2(\mathbb{T}^d))$ into itself, for $T^\ast>0$ sufficiently small. If that's the case, then necessarily $u\equiv 0$ on $[0,T^\ast]$ and then we can iterate the argument to cover the whole $[0,T]$. We treat separately the two terms
\begin{equation}\nonumber
u(t)=\int_0^t P_{t-s}(\nabla(bu))\diff s - \int_0^t P_{t-s}(\text{div}b\, u)\diff s = (I)(t) + (II)(t).
\end{equation}
For the first term, using regularity of the heat kernel and the fractional Sobolev embeddings, we have
\begin{equation}\nonumber\begin{split}
\vert I(t)\vert_{L^2}
& \leq \int_0^t \vert P_{t-s}(\nabla(bu))\vert_{L^2}\diff s\\
& \leq C \int_0^t \Vert P_{t-s}(\nabla(bu))\Vert_{W^{\alpha,r}}\diff s\\
& \leq C \int_0^t (t-s)^{-(1+\alpha)/2} \Vert bu\Vert_{L^r}\diff s\\
& \leq C \Vert u\Vert_{L^\infty(0,t;L^2)} \int_0^t (t-s)^{-(1+\alpha)/2} \Vert b\Vert_{L^{q_1}}\diff s
\end{split}\end{equation}
where $\frac{1}{r}=\frac{1}{q_1}+\frac{1}{2}$ and $W^{\alpha,r}\hookrightarrow L^{\tilde{r}}$, $\frac{1}{\tilde{r}} = \frac{1}{r}-\frac{\alpha}{d}$, $\tilde{r}\geq 2$ for some $\alpha<1$ thanks to \eqref{sec 3 - KR type condition on b}. Young's convolution inequality then gives
\begin{equation}\nonumber
\Vert I\Vert_{L^\infty(0,T^\ast;L^2)}\leq C_1 \Vert u\Vert_{L^\infty(0,T^\ast;L^2)}\, \Vert b\Vert_{L^{p_1}(0,T;L^{q_1})}\, \bigg( \int_0^{T^\ast} s^{-p_1^\ast(1+\alpha)/2}\diff s\bigg)^{1/p_1^\ast},
\end{equation}
where $p_1^\ast$ denotes the conjugate exponent of $p_1$. The last quantity is finite if we can take $\alpha$ such that $p_1^\ast(1+\alpha)<2$, which is guaranteed by \eqref{sec 3 - KR type condition on b}. In the case of $(II)(t)$ the calculations are similar, with only a slight difference in the initial part; they lead to
\begin{equation}\nonumber\begin{split}
\Vert & II\Vert_{L^\infty(0,T^\ast;L^2)}\\ 
& \leq C_2 \Vert u\Vert_{L^\infty(0,T^\ast;L^2)}\, \Vert \text{div}b\Vert_{L^{p_2}(0,T;L^{q_2})}\, \bigg( \int_0^{T^\ast} s^{-p_2^\ast(1+\alpha_2)/2}\diff s\bigg)^{1/p_2^\ast}
\end{split}\end{equation}
for a suitable $\alpha_2$ such that the integral is finite. In particular, this shows that for some $T^\ast$ small enough, the map is a contraction and this concludes the proof.
\end{proof}
\begin{remark}
Up to slight modifications, it can be shown with the same type of proof that under \eqref{sec 3 - KR type condition on b}, uniqueness holds also in the class of weak solutions $u\in L^r(0,T;L^2(\mathbb{T}^d))$, with the additional condition $p_1, p_2\geq r^\ast$. Observe that condition $p_1> 2q_1/(q_1-d)$ may be rewritten as
\begin{equation}\nonumber
\frac{2}{p_1}+\frac{d}{q_1}<1,
\end{equation}
which is known in literature as Krylov-R\"ockner condition, see \cite{Kry},\cite{BecFla}.
\end{remark}

The proof of Lemma \ref{lemma sec 3 - sufficient condition for uniqueness of parabolic problem} is standard (it is a slight improvement of the one contained in \cite[Lemma 3.2]{Mau}, which is restricted to the case of time independent $b$) but we had to provide it mainly for two reasons. The first one is that a major part of the results in the literature are set in $\mathbb{R}^d$ and not in $\mathbb{T}^d$; the second and most important one is that usually uniqueness for \eqref{sec 3 - parabolic Cauchy problem, useful repetition} is proved among solutions in a more regular class, tipically $H^p_{2,q}:=L^p(0,T;W^{2,q})\cap W^{1,p}(0,T;L^q)$, see \cite{Kry2} and the appendix of \cite{Kry}. If $u$ belongs in this class, then $\nabla u\in L^\infty([0,T];L^\infty)$ and so there is no need to impose conditions on $\text{div}\,b$. Here however, since our solution $u$ is obtained as the limit of solutions of transport equations, we cannot infer that it belongs to $H^p_{2,q}$, which is why we need to impose the stronger condition \eqref{sec 3 - KR type condition on b}. Maybe further improvements can be done (for instance if both conditions \eqref{sec 2.3 - assumption 2 on b} and\eqref{sec 3 - KR type condition on b} are imposed, then \eqref{sec 2.3 - assumption 1 on b} can be dropped) but we believe the result to be fairly optimal; indeed the Krylov-R\"ockner (KR) condition arises naturally as the subcritical regime of a scaling argument and reaching the critical case (usually referred to as Ladyzhenskaya-Prodi-Serrin condition, (LPS) for short)
\begin{equation}\nonumber
\frac{2}{p_1}+\frac{d}{q_1}=1
\end{equation}
is in general very difficult and seems out of reach in a class of functions like $L^\infty(0,T;L^2)$. For more details on the topic (both the scaling argument and the critical regime) we refer to \cite{BecFla} and the references therein.
\section{Discussion of existence and uniqueness}\label{section 4 - discussion of existence and uniqueness}

%
%
In order for the statement of Theorem \ref{theorem sec 3 - main result} to be non vacuous, we discuss in this section existence and uniqueness of energy solutions, even if it is not the main aim of this paper. Existence is accomplished by a standard Galerkin scheme; regarding uniqueness, several references are given, as well as a proof in the special case $b=0$, but a full answer is missing. We stress however that the statement of the main result holds for \textit{any} sequence of energy solutions, regardless of their uniqueness; indeed the strength of the result also relies on the fact that the limit satisfies an a priori much better posed equation than the approximating sequence.\\
As in the previous section, we consider an a priori given filtered probability space $(\Omega,\mathcal{F},\mathcal{F}_t,\mathbb{P})$ with an $\mathcal{F}_t$-adapted noise $W$, namely we work in the framework of strong solutions in the probabilistic sense.
\subsection{Existence of energy solutions}\label{subsection 4.1 - existence of energy solutions}
%
%
In this subsection, the existence of energy solutions for any initial data $u_0\in L^2(\mathbb{T}^d)$ is shown. The proof is standard and based on a Galerkin approximation scheme.

First we need some preparations. Throughout the proof we will adopt the following notation: by $L^2(\diff \mathbb{P}\otimes \diff t; L^2)$ denotes the space of all $L^2$-valued, square integrable (in the Bochner sense) functions defined on $\Omega\times [0,T]$, endowed with the product $\sigma$-algebra $\mathcal{F}\otimes \mathcal{B}([0,T])$ and the product measure $\diff\mathbb{P}\otimes \diff t$. $L^2(\diff \mathbb{P}\otimes \diff t; L^2)$ is a separable Hilbert space with the scalar product
\begin{equation}\nonumber
\langle f, g\rangle = \int_0^T\mathbb{E}[\langle f(t), g(t)\rangle_{L^2}]\diff t.
\end{equation}
Morover it's reflexive and closed balls are weakly compact, due to its Hilbert space structure, see \cite[Proposition 5.1]{Bre}. Also recall that under weak continuity assumptions (which are satisfied by energy solutions by definition), $\mathcal{F}_t$-adapted processes are actually predictable, namely measurable with respect to the sub-$\sigma$-algebra $\mathcal{P}$ of predictable sets, see \cite[Porposition 3.7]{DaP}. In particular, predictable processes form a closed subspace of $L^2(\diff \mathbb{P}\otimes \diff t; L^2)$ and therefore they are also closed with respect to weak convergence.
%
%
\begin{theorem}\label{theorem sec 4.1 - existence of energy solutions}
Let $b$ satisfy \eqref{sec 2.3 - assumption 1 on b} and \eqref{sec 2.3 - assumption 2 on b}, $\{\theta_k\}_k$ satisfy \eqref{sec 2.2 - condition on the coefficients} and \eqref{sec 2.2 - isotropy condition} and $W$ be the associated divergence free noise. Then for any $u_0\in L^2$ there exists an energy solution $u$ of \eqref{sec 2.3 - equation in Stratonovich compact form, useful repetition}.
\end{theorem}
\begin{proof}
For any $N>0$, let $\Pi_N$ denote the the Fourier projector on the modes with magnitude $\vert k\vert\leq N$; define $u_0^N = \Pi_N u_0$. For any $N$, consider the following Cauchy problem:
\begin{equation}\label{sec 4.1 - Galerkin approximation system inside existence theorem}
\begin{cases}
\diff u^N = \Pi_N (b\cdot \nabla u^N)\diff t +c\Delta u^N\diff t + \Pi_N\left(\sum_{j,k} \theta_k\, e_k\, a_k^{(j)}\cdot \nabla u^N\diff W^{(j)}_k \right)\\
u^N(0)=u_0^N
\end{cases}
\end{equation}
It can be checked, by writing explicitly the Fourier decomposition, that the above system only involves the noises $W_k^{(j)}$ belonging to a finite set of indices. It is therefore a linear SDE defined on a finite dimensional space (the space of Fourier polynomials of degree at most $N$) and as such it admits a unique local solution $u^N$ with continuous paths. Moreover, the heuristic calculation regarding the energy balance done in Section \ref{subsection 2.3 - SPDE in Ito form and definition of energy solutions} in this setting is actually rigorous, since we are summing over a finite series, and therefore $u^N$ is defined on the whole $[0,T]$ and satisfies
\begin{equation}\nonumber
\sup_{t\in [0,T]}\Big\{ \vert u^N(t)\vert^2_{L^2}\, e^{-\int_0^t \Vert \text{div}b(s,\cdot)\Vert_\infty\diff s} \Big\}\leq \vert u^N_0\vert^2_{L^2} \quad\mathbb{P}\text{-a.s.};
\end{equation}
in particular, for any $N$,
\begin{equation}\label{sec 4.1 - energy inequality inside existence theorem}
\vert u^N(\omega,t)\vert^2_{L^2}\leq e^{\int_0^t \Vert \text{div}b(s,\cdot)\Vert_\infty\diff s} \vert u_0\vert^2_{L^2} \quad\text{ for }(\diff\mathbb{P}\otimes \diff t)\text{-a.e. }(\omega,t).
\end{equation}
This implies that for any $N$, equation \eqref{sec 4.1 - Galerkin approximation system inside existence theorem} has a unique solution, globally defined on $[0,T]$, and
\begin{equation}\nonumber
\int_0^T \mathbb{E}\big[\vert u^N(t)\vert_{L^2}^2\big] \diff t
\leq K \vert u_0\vert^2_{L^2},
\end{equation}
for a suitable constant $K$ which only depends on $b$. Therefore the sequence $\{u^N\}_N$ is uniformly bounded in $L^2(\diff\mathbb{P}\otimes \diff t; L^2)$ and we can assume, up to extracting a (not relabelled) subsequence, that it weakly converges to a process $u$. We now proceed to show that there exists a version of $u$ which is a weak solution of \eqref{sec 2.3 - equation in Stratonovich compact form, useful repetition} with initial data $u_0$. As recalled earlier, $u$ is a predictable process since $u^N$ are so. Fix $\varphi\in C^\infty(\mathbb{T}^d)$, then by testing $u^N$ against $\varphi$ we find
\begin{equation}\label{sec 4.1 - weak formulation for u^N inside existence theorem}\begin{split}
\langle u^N(\cdot),\varphi\rangle - \langle u^N_0,\varphi\rangle = & -\int_0^\cdot \langle u^N(s), \text{div}(b\Pi_N\varphi)\rangle\diff s + c\int_0^\cdot \langle u^N(s),\Delta \varphi\rangle\diff s\\
& - \sum_{j,k} \theta_k\, \int_0^\cdot \langle u^N(s), e_k\, a_k^{(j)}\cdot \nabla \Pi_N\varphi\rangle\diff W^{(j)}_k(s).
\end{split}\end{equation}
It's clear that $u^N_0\to u_0$ in $L^2$; the map $u(\cdot)\mapsto \langle u(\cdot),\varphi\rangle$, from $L^2(\diff\mathbb{P}\otimes \diff t; L^2)$ to $L^2(\diff\mathbb{P}\otimes \diff t)$ is linear and continuous and thus also weakly continuous (this is an immediate consequence of the definition of weak convergence). Similarly, the map from $L^2(\diff\mathbb{P}\otimes \diff t; L^2)$ to $L^2(\diff\mathbb{P}\otimes \diff t)$ given by
\begin{equation}\nonumber
u(\cdot)\mapsto \int_0^\cdot \langle u(s), \text{div}(b\varphi)\rangle\diff s
\end{equation}
is linear and continuous since
\begin{equation}\nonumber\begin{split}
\int_0^T & \mathbb{E}\Bigg[\bigg\vert \int_0^t \langle u^N(s), \text{div}(b\varphi)\rangle\diff s\bigg\vert^2\Bigg]\diff t\\
& \leq T\int_0^T \vert \text{div}(b(t)\varphi)\vert_{L^2}^2\diff t\, \int_0^T \mathbb{E}[\vert u(s)\vert_{L^2}^2]\diff t
\end{split}\end{equation}
and therefore also weakly continuous; since we also have $\text{div}(b\Pi^N\varphi)\to \text{div}(b\varphi)$ strongly, the above estimate shows that overall
\begin{equation}\nonumber
\int_0^\cdot \langle u^N(s), \text{div}(b\Pi_N\varphi)\rangle\diff s \rightharpoonup \int_0^\cdot \langle u(s), \text{div}(b\varphi)\rangle\diff s\quad \text{weakly}.
\end{equation}
A similar reasoning applies to the processes $\int_0^\cdot \langle u^N(s),\Delta\varphi\rangle\diff s$. Regarding the stochastic integrals, again we have that for fixed $\varphi$ the map from $L^2(\diff\mathbb{P}\otimes \diff t; L^2)$ to $L^2(\diff\mathbb{P}\otimes \diff t)$ given by
\begin{equation}\nonumber
u\mapsto \sum_{j,k} \theta_k\, \int_0^\cdot \langle u^N(s), e_k\, a_k^{(j)}\cdot \nabla \Pi_N\varphi\rangle\diff W^{(j)}_k(s)
\end{equation}
is linear and continuous, since by the same calculations of Section \ref{subsection 2.3 - SPDE in Ito form and definition of energy solutions} it holds
\begin{equation}\nonumber\begin{split}
\int_0^T & \mathbb{E}\Bigg[\Big\vert \sum_{j,k} \theta_k\, \int_0^t \langle u^N(s), e_k\, a_k^{(j)}\cdot \nabla \Pi_N\varphi\rangle\diff W^{(j)}_k(s)  \Big\vert^2\Bigg]\diff t\\
& \leq 2T\,\sup_k \theta_k^2\, \Vert\nabla\varphi\Vert_{\infty}^2\,\int_0^T \mathbb{E}[\vert u(s)\vert_{L^2}^2]\diff t
\end{split}\end{equation}
and as before using the fact that $\nabla\Pi_N\varphi\to\nabla\varphi$ uniformly, we also obtain weak convergence. Taking the weak limit as $N\to\infty$ on both sides of \eqref{sec 4.1 - weak formulation for u^N inside existence theorem} we conclude that $u$ satisfies 
\begin{equation}\label{sec 4.1 - weak formulation, useful repetition inside existence theorem}
\begin{split}
\langle u(\cdot),\varphi\rangle - \langle u_0,\varphi\rangle = & -\int_0^\cdot \langle u(s), \text{div}(b\varphi)\rangle\diff s + c\int_0^\cdot \langle u(s),\Delta \varphi\rangle\diff s\\
& - \sum_{j,k} \theta_k\, \int_0^\cdot \langle u(s), e_k\, a_k^{(j)}\cdot \nabla \varphi\rangle\diff W^{(j)}_k(s),
\end{split}
\end{equation}
so that $u$ is a candidate weak solution of \eqref{sec 2.3 - equation in Stratonovich compact form, useful repetition}.
%
%
We now want to show that there exists a version of $u$ which has paths in $C([0,T];L^2_w)$ and satisfies the energy inequality. Observe that the collection of processes in $L^2(\diff\mathbb{P}\otimes \diff t; L^2)$ satisfying inequality \eqref{sec 4.1 - energy inequality inside existence theorem} is a convex, closed subset. Therefore it is also weakly closed (see \cite{Bre}), which implies that inequality \eqref{sec 4.1 - energy inequality inside existence theorem} holds also for $u$.

For fixed $\varphi\in C^\infty(\mathbb{T}^d)$, by standard properties of Lebesgue integral we know that the processes $\int_0^\cdot \langle u(s), \text{div}(b\varphi)\rangle \diff s$, $\int_0^\cdot \langle u(s),\Delta \varphi\rangle \diff s$ are $\mathbb{P}$-a.s. continuous. 
Recall that by construction of the It\^o integral, for any $k$ and $j$ the process $\int_0^\cdot \langle u(s), e_k\, a_k^{(j)}\cdot \nabla \varphi\rangle\diff W^{(j)}_k(s)$ is a continuous square integrable martingale; moreover, continuous square integrable martingales are closed under $L^2(\diff \mathbb{P}\otimes \diff t)$-convergence (see \cite{Rev}) and we have already shown that the infinite series on the r.h.s. of \eqref{sec 4.1 - weak formulation, useful repetition inside existence theorem} is convergent in this norm. 
Therefore we can conclude that, for a fixed $\varphi\in C^\infty(\mathbb{T}^d)$, the process appearing on the r.h.s. of \eqref{sec 4.1 - weak formulation, useful repetition inside existence theorem} is $\mathbb{P}$-a.s. continuous in time and it coincides up to $(\diff \mathbb{P}\otimes \diff t)$-negligible sets with $\langle u,\varphi\rangle-\langle u_0,\varphi\rangle$; in particular, $\mathbb{P}$-a.s. $\langle u,\varphi\rangle$ admits a continuous version.

But then we can find, also thanks to \eqref{sec 4.1 - energy inequality inside existence theorem}, a countable dense collection $\varphi_n$ and a subset $\Gamma$ of $\Omega$ with $\mathbb{P}(\Gamma)=1$ such that for all $\omega\in\Gamma$ the following holds: $t\mapsto\langle u(\omega,t),\varphi_n\rangle$ admits a continuous version for all $n$ and there exists a set $E_\omega\subset [0,T]$ of full measure on which $\vert u(\omega,\cdot)\vert_{L^2}$ is uniformly bounded by some constant $K$. In particular, for any $t\notin E_\omega$ and any sequence $\{t_n\}_n\subset E_\omega$ such that $t_n\to t$, we can extract a subsequence such that $u(\omega, t_n)$ admits a weak limit in $L^2$, denoted by $v(\omega,t)$, whose norm is still bounded by the constant $K$. But since $\langle u(\omega,\cdot),\varphi_n\rangle$ all admit continuous versions, the limit $v(\omega,t)$ is uniquely determined and does not depend on the extracted subsequence, nor on the original sequence $\{t_n\}_n$. Reasoning in this way, for fixed $\omega\in\Gamma$, we can define $v(\omega,t)$ for all $t\in [0,T]$ and it satisfies the following: $v(\omega,t)=u(\omega,t)$ for all $t$ in a set of full Lebesgue measure; $\vert v(\omega,t)\vert_{L^2}\leq K$ for all $t\in [0,T]$; $\langle v(\omega,t),\varphi_n\rangle$ coincides with the continuous version of $\langle u(\omega,t),\varphi_n\rangle$. But then by the uniform bound and density of $\varphi_n$ it follows that the map $t\mapsto \langle v(\omega,t)\varphi\rangle$ is continuous for every $\varphi\in C^\infty(\mathbb{T}^d)$ and for every $\omega\in \Gamma$, namely $v$ is a version of $u$ with $\mathbb{P}$-a.s. weakly continuous paths. Since $v$ also satisfies \eqref{sec 4.1 - weak formulation, useful repetition inside existence theorem}, we conclude that it is a weak solution.

It only remains to show that the energy inequality holds, but this is achieved similarly by using the fact that, for all $\omega\in \Gamma$ and all $t\in \tilde{E}_\omega$, $\tilde{E}_\omega$ being a full Lebesgue measure set, inequality \eqref{sec 4.1 - energy inequality inside existence theorem}  indeed holds, and therefore by using lower semicontinuity of $\vert\cdot\vert_{L^2}$ and $v(\omega,\cdot)\in C([0,T];L^2_w)$ it can be extended to all $(\omega,t)\in \Gamma\times[0,T]$.
\end{proof}

\subsection{Proof of pathwise uniqueness in the case $b=0$}\label{subsection 4.2 - proof of strong uniqueness in the case b=0}

%
%

We prove in this section pathwise uniqueness of solutions in the case $b=0$; before proceeding further, let us mention the already existing results in the literature. Many of them are proved in $\mathbb{R}^d$ but can be easily generalized to $\mathbb{T}^d$.

A main result in the topic is the already mentioned work \cite{FlaGub}, where it is shown that for $b\in L^\infty(0,T;C^\alpha(\mathbb{R}^d;\mathbb{R}^d))$ with $\text{div}\, b\in L^p([0,T]\times\mathbb{R}^d)$, $\alpha>0$ and $p\geq 2$, in the case of a space-independent standard $d$-dim. Brownian motion, pathwise uniqueness holds for \eqref{sec 1 - STLE in compact Stratonovich form} for any $u_0\in L^\infty(\mathbb{R}^d)$; the proof is based on the existence of a sufficiently regular flow for the associated SDE. Many other results are now available, see the references in \cite{BecFla}, but tipically the noise considered is space independent or has sufficiently good space regularity. In the stochastic fluid dynamics literature, the use of divergence free transport noise of the form
\begin{equation}\nonumber
W(t,x)=\sum_{n\in\mathbb{N}} \sigma_n(x)W_n(t)
\end{equation}
appears fairly often; typical assumptions on this kind of noise are like those contained in \cite{Brz}, specifically it is required that\begin{equation}\label{sec 4.2 brz condition}
\bigg\Vert \sum_{n\in\mathbb{N}} \vert D\sigma_k(\cdot)\vert^2\bigg\Vert_{L^\infty(\mathbb{T}^d)}<\infty.
\end{equation}
In \cite{BecFla}, \eqref{sec 1 - STLE in compact Stratonovich form} is studied mainly in the case of space-independent noise, sufficiently regular initial data and $b\in L^p(0,T;L^q(\mathbb{R}^d))$, both in the subcritical (KR) and the critical (LPS) regime, using PDE arguments which do not rely on the existence of a regular flow for the associated SDE. It is stated however in Section 1.9 that all the results generalize to the case of \textquotedblleft $\sigma_n$ of class $C^4_b$ with proper summability in $n$ \textquotedblright\ such that the SDE
\begin{equation}\nonumber
\diff Y=\sum_{n\in\mathbb{N}} \sigma_n(Y)\circ \diff W_n
\end{equation}
has a sufficiently regular stochastic flow of diffeomorphisms. In particular we expect that at least an analogue requirement to \eqref{sec 4.2 brz condition} is needed; in our setting, this condition is equivalent to
\begin{equation}\label{sec 4.2 brz condition 2}
\sum_{k\in\mathbb{Z}^d}\vert k\vert^2 \theta_k^2<\infty.
\end{equation}
However, if instead of pathwise uniqueness one only requires \textit{Wiener uniqueness} of weak solutions of \eqref{sec 1 - STLE in compact Stratonovich form}, then the problem greatly simplifies. Wiener uniqueness means uniqueness in the class of processes adapted to the Brownian filtration $\mathcal{F}^W_t$ and can be established by Wiener chaos expansion techniques, see \cite{LeJ}, \cite{LeJ2} and \cite{Mau}. In particular, even if in \cite{Mau} only space-independent noise and time-independent drift are considered, the technique seems to easily adapt to our setting, for any $\{\theta_k\}$ satisfying \eqref{sec 2.2 - condition on the coefficients} and \eqref{sec 2.2 - isotropy condition} and any $b$ satisfying \eqref{sec 2.3 - assumption 2 on b}, \eqref{sec 3 - KR type condition on b}, as it fundamentally only requires wellposedness in a suitable class for the Kolmogorov equation \eqref{sec 3 - parabolic Cauchy problem, useful repetition}, which holds under the conditions of Lemma \ref{lemma sec 3 - sufficient condition for uniqueness of parabolic problem}. Wiener uniqueness is however unsatisfactory, for several reasons: if the only information on a solution $u$ is that it is adapted to $\mathcal{F}_t$, then this strategy only gives pathwise uniqueness of the process $\tilde{u}(t)=\mathbb{E}[u(t)\vert \mathcal{F}^W_t]$; Wiener uniqueness is also too weak to apply the Yamada-Watanabe theorem (which holds also in infinite dimensions, see \cite{Roc}) and ill-suited to exploit tools like Girsanov transform, see the discussion in Section 4.7 of \cite{Fla}.

The problem of passing from Wiener uniqueness to pathwise uniqueness is not only technical, because if condition \eqref{sec 4.2 brz condition} is not satisfied, the equation cannot be in general solved by means of characteristics, since phenomena like splitting and coalescence can occurr, as shown in \cite{LeJ}. In this work, Wiener uniqueness is exploited to construct Markovian statistical solutions $S_t$ which are then studied and classified; in our setting, for $b=0$ and $W$ divergence free, according to the terminology introduced in \cite{LeJ}, the statistical solution is \textit{diffusive without hitting} and is not a flow of maps (i.e. does not admit a representation by characteristics), see Theorem 10.1 (by the divergence free condition, in our setting the parameter $\eta$ is always $1$). In particular splitting can occurr, since the $2$-point motion $(X_t,Y_t)$ associated to $S_t$ starting from $(x,x)$ satisfies $X_t\neq Y_t$ for all positive $t$, see Definition 6.3. In the work \cite{LeJ2} statistical solutions are studied more in depth and it is hinted that non uniqueness can happen only in the \textit{turbulent with hitting} regime, but no explicit proof of pathwise uniqueness in the other regimes is given.

Here instead we adapt the strategy developed in \cite{Barb}, which yields a relatively simple and short proof of pathwise uniqueness in the special case $b=0$, for any $\{\theta_k\}_k$ satisfying \eqref{sec 2.2 - condition on the coefficients} and \eqref{sec 2.2 - isotropy condition}, which is a much weaker condition compared to \eqref{sec 4.2 brz condition 2}.

%
%

We focus on the SPDE
\begin{equation}\label{sec 4.2 - STLE with b=0, Ito form}
\diff u = c\Delta u\diff t + \diff W\cdot\nabla u,
\end{equation}
with $W$ as usual given by \eqref{sec 2.2 - definition of the noise W(t,x)}, $c$ defined in function of $\{\theta_k\}_k$ by \eqref{sec 2.3 - definition of the constant c}. Given $u_0\in L^2$, we consider a weak solution of \eqref{sec 4.2 - STLE with b=0, Ito form} in the sense of Definition \ref{definition sec 2.3 - weak solution}.  Following \cite{Barb}, we rewrite the equation in Fourier components. Let $u$ be given by the Fourier series
\begin{equation}\nonumber
u(t,x)=\sum_l u_l(t)\,e_l(x),
\end{equation}
so that
\begin{equation}\nonumber
\nabla u = i\sum_l l\, u_l\,e_l,
\quad
\Delta u = -\sum_l \vert l\vert^2 u_l\, e_l.
\end{equation}
Explicit calculations give the Fourier expansion for $\diff W\cdot\nabla u$:
\begin{equation}\nonumber\begin{split}
\diff W\cdot\nabla u
& = \Bigg(\sum_{j,k} \theta_k\,e_k\,a_k^{(j)}\diff W_k^{(j)} \Bigg) \cdot\left(i\sum_l l\, u_l\,e_l\right)\\
& = \sum_{j,k,l} i\,\theta_k\, a_k^{(j)}\cdot l\, u_l\, e_{k+l}\diff W_k^{(j)}\\
& = \sum_k i\Bigg(\sum_{j,l}\theta_{k-l}\, a_{k-l}^{(j)}\cdot l\, u_l\diff W^{(j)}_{k-l} \Bigg)\, e_k\\
& = \sum_k i\Bigg(\sum_{j,l}\theta_{k-l}\, a_{k-l}^{(j)}\cdot k\, u_l\diff W^{(j)}_{k-l} \Bigg)\, e_k,
\end{split}\end{equation}
where in the last passage we used the fact that $a^{(j)}_{k-l}\perp k-l$. Uniqueness of the Fourier expansion then gives the following infinite linear system of coupled SDEs for the coefficients $u_k$:
\begin{equation}\label{sec 4.2 - system of SDEs for u_k}
\diff u_k = -c\vert k\vert^2\,u_k\diff t + i \sum_{j,l}\theta_{k-l}\, a_{k-l}^{(j)}\cdot k\, u_l\diff W^{(j)}_{k-l},
\end{equation}
where as usual the identity must be interpreted in integral sense:
\begin{equation}\nonumber
u_k(t) - u_k(0) = -c\vert k\vert^2\int_0^t u_k(s)\diff s + i \sum_{j,l}\theta_{k-l}\,a_{k-l}^{(j)}\cdot k \int_0^t u_l(s)\diff W^{(j)}_{k-l}(s).
\end{equation}
The derivation of \eqref{sec 4.2 - system of SDEs for u_k} was very heuristical, but it can be checked that we would have found the same exact expression in integral form by taking $\varphi=e_k$ as test functions in \eqref{sec 2.3 - equation in Ito form, weak formulation in integral form}. Calculations similar to those of Section \ref{subsection 2.3 - SPDE in Ito form and definition of energy solutions} give
\begin{equation}\nonumber\begin{split}
\mathbb{E}\Bigg[\Big\vert \sum_{j,l}\theta_{k-l}\,a_{k-l}^{(j)}\cdot k \int_0^t u_l(s)\diff W^{(j)}_{k-l}(s) \Big\vert^2\Bigg]
& \leq K \vert k\vert^2 \int_0^T \sum_l \mathbb{E}\big[\vert u_l(s)\vert^2\big] \diff s\\
& = K \vert k\vert^2 \int_0^T \mathbb{E}\big[\vert u(s)\vert_{L^2}^2\big] \diff s.
\end{split}\end{equation}
for a suitable constant $K$, so that the infinite series in \eqref{sec 4.2 - system of SDEs for u_k} is well defined, since $u$ satisfies \eqref{sec 2.3 - condition of bounded energy in the definition of weak solution}.
To prove uniqueness, we need the following result.
%
%
\begin{lemma}\label{lemma sec 4.2 - system for averaged energies}
Let $u$ be a weak solution of \eqref{sec 4.2 - STLE with b=0, Ito form}, $u_k$ defined as above. Then the real functions $x_k$ defined by $x_k(t)=\mathbb{E}[\vert u_k(t)\vert^2]$ satisfy the following linear infinite system of coupled ODEs:
\begin{equation}\label{sec 4.2 - equation for average energies}
\dot{x}_k = -2c\vert k\vert^2 x_k + 2\sum_l \theta_{k-l}^2\, \vert P_{k-l}k\vert^2\,x_l.
\end{equation}
\end{lemma}
\begin{proof}
Since $u$ is a weak solution, we know that $\{u_k\}_k$ satisfy system \eqref{sec 4.2 - system of SDEs for u_k}. Then applying It\^o formula, using the properties of $W^{(j)}_k$, we have
\begin{equation}\nonumber\begin{split}
\diff(\vert u_k\vert^2)
& = \diff(u_k\overline{u}_k)
= u_k\diff \overline{u}_k + \overline{u}_k\diff u_k + \diff [u_k,\overline{u}_k]\\
& = -c\vert k\vert^2 \vert u_k\vert^2\diff t + \diff M_t -c\vert k\vert^2 \vert u_k\vert^2\diff t + \diff N_t \\ & \quad + 2\sum_{j,l} \theta_{k-l}^2 \vert a_{k-l}^{(j)}\cdot k\vert^2 \vert u_l\vert^2\diff t
\end{split}\end{equation}
where $M$ and $N$ are suitable square integrable martingale starting at 0. Taking expectation their contribution disappears and we obtain
\begin{equation}\nonumber
\dot{x}_k = -2c\vert k\vert^2 x_k + 2\sum_{j,l} \theta_{k-l}^2\, \vert a_{k-l}^{(j)}\cdot k\vert^2\,x_l.
\end{equation}
Observe that, for any fixed $k$,
\begin{equation}\nonumber
\sum_j \vert a_{k-l}^{(j)}\cdot k\vert^2
= \Bigg\vert \sum_j \Big(a_{k-l}^{(j)}\otimes a_{k-l}^{(j)}\Big)k \Bigg\vert^2
= \vert P_{k-l} k\vert^2,
\end{equation}
which implies the conclusion.
\end{proof}
%
%
\begin{remark}\label{remark sec 4.2 - forward equation of markov chain} System \eqref{sec 4.2 - equation for average energies} can be written as
\begin{equation}\nonumber
\dot{x}_k = q_{kk}\,x_k + \sum_l q_{kl}\, x_l,
\end{equation}
where $q_{kk}=-2c\vert k\vert^2\in (-\infty, 0)$, $q_{kl} = 2 \theta_{k-l}^2\, \vert P_{k-l}k\vert^2\geq 0$ for $k\neq l$. Moreover, for any $k$ it holds
\begin{equation}\nonumber
\sum_l q_{kl}
= 2 \sum_l \theta_{k-l}^2\, \vert P_{k-l}k\vert^2
= 2 \sum_{\tilde{l}} \theta_{\tilde{l}}^2\, \vert P_{\tilde{l}}k\vert^2
= 2 c\vert k\vert^2 = -q_{kk},
\end{equation}
where we used the change of variables $k-l=\tilde{l}$ and the computation \eqref{sec 2.3 - definition of the constant c} from Section \ref{subsection 2.3 - SPDE in Ito form and definition of energy solutions}. Namely, system \eqref{sec 4.2 - equation for average energies} can be interpreted as the forward equation associated to a $Q$-matrix, which is the generator of a continuous time Markov process on $\mathbb{Z}^d\setminus\{0\}$; see the similarity with \cite{BarFla}. This formulation can be useful to study long-time behaviour of solutions: for instance we can deduce immediately that if $u$ is a stationary solution, then it must hold $x_k=c$ for all $k$ and so the only invariant measure with support in $L^2$ is $\delta_0$. If we expect convergence to equilibrium as $t\to\infty$, then all solutions should converge to $0$, even if energy is a formal invariant for equation \eqref{sec 4.2 - STLE with b=0, Ito form}; indeed in \cite{BarFla} anomalous dissipation of energy for a similar model is shown. Understanding whether anomalous dissipation takes place in this model will be the subject of future research.
\end{remark}
%
%
\begin{theorem}\label{theorem sec 4.2 - uniqueness for b=0} Let $\{\theta_k\}_k$ satisfy \eqref{sec 2.2 - condition on the coefficients} and \eqref{sec 2.2 - isotropy condition} and $W$ be the associated divergence free noise. If $u$ and $v$ are two weak solutions of \eqref{sec 4.2 - STLE with b=0, Ito form} with the same initial data $u_0$, then
\begin{equation}\nonumber
\mathbb{P}\big(u(t)=v(t)\big)=1\quad \forall\,t\in [0,T].
\end{equation}
\end{theorem}
\begin{proof}
By linearity of equation \eqref{sec 4.2 - STLE with b=0, Ito form}, $w:= u-v$ is a weak solution with initial data $w_0=0$. In order to conclude it suffices to show that $\mathbb{P}(w(t)=0)=1$ for every $t$. Recall that by the definition of weak solution, $u$ and $v$ satisfy \eqref{sec 2.3 - condition of bounded energy in the definition of weak solution} and therefore also $w$ does. By Lemma \ref{lemma sec 4.2 - system for averaged energies}, we know that $x_k=\mathbb{E}[\vert w_k\vert^2]$ satisfy \eqref{sec 4.2 - equation for average energies} with initial condition $x_k(0)=0$ for all $k$, namely
\begin{equation}\nonumber
x_k(t) = -2c \vert k\vert^2 \int_0^t x_k(s)\diff s + 2\sum_l \theta_{k-l}^2 \vert P_{k-l} k\vert^2 \int_0^t  x_l(s) \diff s.
\end{equation}
Fix $t\in [0,T]$ and define $A_k=\int_0^t x_k(s)\diff s$, then the above equation becomes
\begin{equation}\label{sec 4.2 - proof of main theorem, internal equation}
x_k(t) + 2c\vert k\vert^2 A_k= 2\sum_l \theta_{k-l}^2 \vert P_{k-l} k\vert^2 A_l.
\end{equation}
Condition \eqref{sec 2.3 - condition of bounded energy in the definition of weak solution} implies that $A_k$ is summable:
\begin{equation}\nonumber
\sum_k A_k
= \sum_k \int_0^t \mathbb{E}[\vert w_k(s)\vert^2]\diff s
\leq \int_0^T \mathbb{E}[\vert w(s)\vert_{L^2}^2]\diff s<\infty,
\end{equation}
so that $A_k\to 0$ as $k\to\infty$; in particular $\{A_k\}_k$ admits a maximum, say at $A_{k_1}$. Since $x_{k_1}(t)\geq 0$ by construction, we find
\begin{equation}\nonumber\begin{split}
2c\vert k_1\vert^2 A_{k_1}
& \leq x_{k_1}(t) + 2c\vert k_1\vert^2 A_{k_1}
= 2\sum_l \theta_{k_1-l}^2 \vert P_{k_1-l} k_1\vert^2 A_l\\
& \leq 2\max_l A_l \sum_l \theta_{k_1-l}^2 \vert P_{k_1-l} k_1\vert^2
= 2c\vert k_1\vert^2 A_{k_1},
\end{split}\end{equation}
which implies that all the inequalities are equalities and therefore $x_{k_1}(t)=0$, $A_l=A_{k_1}$ for all $l$ such that $\theta_{k_1-l}\neq 0$ and $P_{k_1-l}k_1\neq 0$ (i.e. $l\notin \langle k\rangle$). We are now going to show that we can construct inductively a sequence $k_n$ such that $k_n\to\infty$ and $A_{k_n}=\max_l A_l$. Recall that $\{\theta_k\}_k$ satisfy the isotropy condition \eqref{sec 2.2 - isotropy condition}, so if $\theta_{\overline{j}}\neq 0$ for some $\overline{j}$, then $\theta_{O\overline{j}}\neq 0$ as well. Let $\Gamma=\{O\overline{j}, O\in E_{\mathbb{Z}}(d)\}$. Then we can find $j\in\Gamma$ such that $k_2:=k_1-j$ satisfies $k_2\notin\langle k_1\rangle$ and $\vert k_2\vert> \vert k_1\vert$; since $\theta_{k_1-k_2}=\theta_j\neq 0$ we conclude that $A_{k_2}=\max_l A_l$. But then we can iterate the reasoning, this time starting from $A_{k_2}$, to find $k_3$ such that $\vert k_3\vert>\vert k_2\vert$ and $A_{k_3}=\max_l A_l$, and so on. In this way we find the desidered sequence $\{k_n\}_n$; but $A_l\to 0$ as $l\to\infty$, which implies that $\max_l A_l=0$ and so $A_l=0$ for all $l$.
Since
\begin{equation}\nonumber
0=\sum_l A_l = \int_0^t \mathbb{E}[\vert w(s)\vert^2_{L^2}]\diff s
\end{equation}
and the reasoning holds for any $t\in [0,T]$, we obtain the conclusion.
\end{proof}
\begin{remark}\label{remark sec 4.2 - advantages and disadvantages of this line of proof}
Let us underline both the advantages and the disadvantages of the approach we used. On one side,  the proof could be further generalised: if we had a weaker concept of solution for which the derivation of system \eqref{lemma sec 4.2 - system for averaged energies} is still rigorous (in principle system \eqref{sec 4.2 - equation for average energies} is well defined under assumption $\{x_k\}_k\in l^\infty$), then in order for the proof to work we only need to guarantee that $\{A_k\}\in c_0$ (i.e. $\{A_k\}\in l^\infty$ and $A_k$ infinitesimal), which could be deduced under milder conditions than \eqref{sec 2.3 - condition of bounded energy in the definition of weak solution}. The proof also holds for the inhomogeneous equation with an external forcing $f$, since the difference of two solutions of the inhomogeneous system is a solution of the homogenous one.\\
On the other side, the proof is not easily generalizable on domains different than $\mathbb{T}^d$ and completely breaks down when treating the case $b\neq 0$. In fact, we are not able to obtain a closed equation for $\mathbb{E}[\vert u_k\vert^2]$ as in Lemma \ref{lemma sec 4.2 - system for averaged energies} anymore; it's still possible to find a closed system of ODEs for the terms $x_{k,l}=\mathbb{E}[u_k\overline{u}_l]$, but it's not as nice as \eqref{sec 4.2 - equation for average energies}. The simplification obtained by finding a closed equation for the "diagonal" terms $x_{k,k}=\mathbb{E}[\vert u_k\vert^2]$ is the key in our method of proof.
\end{remark}

We immediately obtain the following corollary.

\begin{corollary}\label{corollary sec 4.2 - pathwise uniqueness} The following hold.
\begin{itemize}

\item[i)] (Pathwise uniqueness) Let $u$ be a weak solution of \eqref{sec 4.2 - STLE with b=0, Ito form} with initial data $u_0$. Then $u$ is the unique energy solution of \eqref{sec 4.2 - STLE with b=0, Ito form} with initial data $u_0$ (up to indistinguishability).

\item[ii)] (Stability) Let $u$ and $v$ be two weak solutions with respect to initial data $u_0$ and $v_0$. Then
\begin{equation}\nonumber
\Vert u-v\Vert_{L^\infty(0,T;L^2)}  \leq \vert u_0-v_0\vert_{L^2}\quad \mathbb{P}\text{-a.s.}
\end{equation}

\end{itemize}
\end{corollary}
\begin{proof}
i) Let $u$ be as in the hypothesis and $\tilde{u}$ be an energy solution of \eqref{sec 4.2 - STLE with b=0, Ito form} with initial data $u_0$. Then by Theorem \ref{theorem sec 4.2 - uniqueness for b=0}
\begin{equation}\nonumber
\mathbb{P}(u(t)=\tilde{u}(t)\ \ \forall\, t\in [0,T]\cap\mathbb{Q})=1
\end{equation}
and we conclude that $u$ and $\tilde{u}$ are indistinguishable, since $u$ and $\tilde{u}$ both have $\mathbb{P}$-a.s. continuous paths in the weak topology.

ii) By linearity, $u-v$ is a weak solution with initial data $u_0-v_0$. But then it is the unique energy solution and satisfies the energy inequality \eqref{sec 2.3 - energy inequality} with $b=0$, which gives the conclusion.
\end{proof}

Also observe that we are in the conditions to apply Yamada-Watanabe theorem and therefore not only pathwise but also uniqueness of strong solutions in the probabilistic sense holds; this also implies uniqueness in law.

\section{The case $d=1$}\label{section 5 - the case d=1}
%
%
The proof of our main result required two fundamental features: the use of an incompressible noise, namely $W=W(t,x)$ such that at any fixed time $\text{div}_x W(t,\cdot)=0$ in the sense of distributions, and the existence of weak solutions satisfying a uniform energy bound; the existence of such solutions was a consequence of suitable conditions on $b$ and incompressibility of $W$. However, in the case of spatial dimension $d=1$, the divergence-free condition is equivalent to $W$ being space-independent, which doesn't allow to take a sequence of noises which are increasingly rougher in space; it is therefore unclear whether it's possible to obtain an analogue of Theorem \ref{theorem sec 3 - main result}. One might still look for a suitable sequence of not divergence-free noises, for which existence of energy solutions can be shown, and such that in the limit they converge in a suitable sense to a deterministic PDE. In this section we show that, while this program can still partially be carried out, in the limit we do not expect to find a PDE with a diffusion term.

%
%

Let us briefly introduce the notation and setting of this section. We consider the one dimensional torus $\mathbb{T}=\mathbb{R}/(0,2\pi)$, with periodic boundary condition. However, since in one dimension there is no real advantage in working with complex series, we prefer in this case to restrict ourselves to the space of real valued, $2\pi$-periodic, $L^2$ functions, with the normalised inner product
\begin{equation}\nonumber
\langle f,g\rangle = \frac{1}{2\pi}\int_\mathbb{T} f(x)g(x)\diff x.
\end{equation}
A complete orthonormal system for $L^2(\mathbb{T})$ is given by the real Fourier basis $\{e_k\}_{k\in\mathbb{Z}}$:
\begin{equation}\nonumber
e_k(x)=\begin{cases}
\sqrt{2}\cos(kx)\qquad & \text{if } k>0\\
1 & \text{ if } k=0\\
\sqrt{2}\sin(kx) & \text{ if } k<0
\end{cases}.
\end{equation}
Throughout this section, we will assume $\{W_k\}_{k\in\mathbb{Z}}$ to be a sequence of \textit{real} independent Brownian motions. In order to understand which kind of noise to use, let us start by considering the deterministic transport equation
\begin{equation}\label{sec 5 - starting deterministic pde}
\partial_t u = b\, \partial_x u,
\end{equation}
where we need to assume at least $b=b(t,x)$ to be in $L^1(0,T;L^2)$ with $\text{div}\, b = \partial_x b \in L^1(0,T;L^\infty)$, namely $b\in L^1(0,T; W^{1,\infty})$. In the 1-d framework, under such conditions on $b$, equation \eqref{sec 5 - starting deterministic pde} is already well posed by the DiPerna-Lions theory, see \cite{DiPLio}. We highlight this aspect as it provides a partial explanation of the fact that in 1-d there doesn't seem to be much space for regularization by space-time dependent noise (even if some results have been obtained, see \cite{Ges} and the references therein): the deterministic theory is already well posed under sufficiently mild conditions and therefore too "competitive" for the noise to perform better. Anyway, observing that for a given smooth deterministic $v=v(t,x)$, the linear PDE
\begin{equation}\nonumber
\partial_t u = 2v\,\partial_x u + \partial_x v\, u
\end{equation}
is formally energy preserving, since
\begin{equation}\nonumber
\frac{\diff}{\diff t}\int_\mathbb{T} u^2\diff x
= \int_\mathbb{T} \big(4 v\,u\,\partial_x u + 2\partial_xv\, u^2\big)\diff x
= 2\int_\mathbb{T} \partial_x (v\, u^2)\diff x = 0,
\end{equation}
it seems reasonable to consider a perturbation of \eqref{sec 5 - starting deterministic pde} of the form
\begin{equation}\label{sec 5 - linear spde, compact stratonovich formulation}
\diff u = b\, \partial_x u\diff t + 2\circ \diff W\, \partial_x u + \circ \diff(\partial_x W)\, u,
\end{equation}
where $W=W(t,x)$ is a suitable noise which will be described later. As before, $\circ \diff W$ denotes Stratonovich integration with respect to the time parameter; observe that, if $W$ is a distribution valued Wiener process, then $\partial_x W$ is again a  a distribution valued Wiener process and therefore, under proper conditions, Stratonovich integration also with respect to it can be defined. This will become more transparent once we describe $W$ explicitly. By the properties of Stratonovich integral (in particular the chain rule) and the above computation, we expect formally to obtain the same energy balance as for equation \eqref{sec 5 - starting deterministic pde} and therefore the existence of weak solutions satisfying the energy inequality
\begin{equation}\nonumber
\sup_{t\in [0,T]} \Big\{ e^{-1/2 \int_0^t \Vert \partial_x b(s)\Vert_\infty \diff s}\, \vert u(t)\vert_{L^2} \Big\} \leq \vert u_0\vert_{L^2} \quad \mathbb{P}\text{-a.s.}
\end{equation}
Observe however that we have already found a criticality with respect to the approach of the previous sections: in equation \eqref{sec 5 - linear spde, compact stratonovich formulation}, not only Stratonovich multiplicative noise $\circ \diff W$ appears, but also $\circ \diff(\partial_x W)$; if the former, in order to be defined, required $W$ to be an $L^2$-valued random variable, then the latter will require $\partial_x W$ to be $L^2$-valued as well and therefore $W$ to belong to $H^1$. In particular, the class of noises we can use has one additional degree of regularity with respect to the one we could use in higher dimension. If we expect the paradigm \textquotedblleft the rougher the noise, the better the regularization\textquotedblright\ to hold, then this kind of noise shouldn't be able to regularize very much.

%
%

Let us give a more precise description of $W$, so that we can give proper meaning to equation \eqref{sec 5 - linear spde, compact stratonovich formulation} and pass to the It\^o formulation. Similarly to the previous sections, we consider $W$ given by
\begin{equation}\nonumber
W(t,x)=\sum_k \sigma_k(x) W_k(t),
\end{equation}
where $\sigma_k\in C^\infty(\mathbb{T})$ for each $k$ and the index $k$ is ranging over $\mathbb{Z}$. In particular, we might consider both a finite or infinite series (in the latter case, convergence is interpreted in the sense of distributions as before). Then, $\partial_x W$ in the sense of distributions is given by
\begin{equation}\nonumber
\partial_x W(t,x)=\sum_k \sigma'_k(x) W_k(t),
\end{equation}
so that we can write equation \eqref{sec 5 - linear spde, compact stratonovich formulation} as
\begin{equation}\label{sec 5 - linear spde, explicit stratonovich formulation}
\diff u
= b\, \partial_x u\diff t + \sum_k \big( 2\sigma_k \partial_x u + \sigma_k'\, u\big)\circ \diff W_k
= b\, \partial_x u\diff t + \sum_k \mathcal{M}_k u\circ \diff W_k,
\end{equation}
where we consider $\mathcal{M}_k = 2\sigma_k \partial_x + \sigma_k'$ as a linear (unbounded) operator on $L^2(\mathbb{T})$. From \eqref{sec 5 - linear spde, compact stratonovich formulation} we obtain the corresponding It\^o formulation:
\begin{equation}\nonumber
\diff u = b\, \partial_x u\diff t +\frac{1}{2} \sum_k \mathcal{M}_k^2 u\diff t + \sum_k \mathcal{M}_k u\diff W_k.
\end{equation}
Algebraic computations yield
\begin{equation}\label{sec 5 - linear spde, general ito formulation}\begin{split}
\diff u = b\, \partial_x u\diff t
& + \Bigg(\sum_k \sigma_k \sigma''_k + \frac{1}{2} (\sigma_k')^2\Bigg) u\diff t
+ 4 \Bigg(\sum_k \sigma_k \sigma'_k \Bigg)\partial_x u\diff t\\
&+ 2\Bigg(\sum_k \sigma_k^2 \Bigg)\partial_{xx} u\diff t
+ \sum_k \mathcal{M}_k u\diff W_k.
\end{split}\end{equation}
In particular, in order for the It\^o-Stratonovich corrector to make sense, we need all the above series to be convergent at any fixed $x$. Moreover, observe that now the term in front of $\partial_{xx}u$ is in \textquotedblleft competition\textquotedblright\ with those in front of $\partial_x u$ and $u$; in particular, when we renormalise the noise $W$ by dividing by the term which explodes faster, if the other terms don't grow with the same speed they will disappear in the limit. This is indeed what happens and what determines the failure of Theorem \ref{theorem sec 3 - main result} in $d=1$, at least when the perturbation is performed by a linear multiplicative noise of the form \eqref{sec 5 - linear spde, compact stratonovich formulation}.

%
%

We illustrate what described above for specific choices of $\sigma_k$ which allow to perform explicit calculations. Take $\sigma_k(x) = \lambda_k\, e_k(x)$, where $\lambda_k$ are some real constants (on which we need to impose conditions, see below) such that $\lambda_k = \lambda_{-k}$ for all $k$, $\lambda_0=0$ and $\{e_k\}_k$ is the real Fourier basis introduced at the beginning of the section. Then equation \eqref{sec 5 - linear spde, general ito formulation} becomes
\begin{equation}\label{sec 5 - linear spde, specific ito formulation}\begin{split}
\diff u
& = b\, \partial_x u \diff t -\frac{1}{2}\Bigg( \sum_k k^2 \lambda_k^2\Bigg) u\diff t + 2\Bigg( \sum_k \lambda_k^2\Bigg)\partial_{xx}u\diff t\\
&\quad+ \sum_k \lambda_k \big( 2 e_k \partial_x u + e_k' u\big) \diff W_k .
\end{split}\end{equation}
Equation \eqref{sec 5 - linear spde, specific ito formulation} confirms the discussion above: the It\^o-Stratonovich corrector in order to be defined requires the condition
\begin{equation}\nonumber
\sum_k k^2 \lambda_k^2 <\infty,
\end{equation}
namely $W$ taking values in $H^1$, and the coefficient in front of $u$ is strictly bigger than the one in front of $\partial_{xx}u$, at least whenever $\lambda_k\neq 0$ for some $k\notin \{-1,1\}$. Let us write the weak formulation of equation \eqref{sec 5 - linear spde, specific ito formulation} in order to understand if the infinite series of It\^o integrals is well defined as well and how fast it grows as a function of the parameters $\lambda_k$: $u$ is a weak solution if, for any $\varphi\in C^\infty(\mathbb{T})$,
\begin{equation}\label{sec 5 - specific ito formulation, weak formulation}\begin{split}
\diff \langle u,\varphi \rangle =
& - \langle u,\partial_x (b \varphi) \rangle\diff t
- \frac{1}{2}\Bigg( \sum_k k^2 \lambda_k^2\Bigg) \langle u,\varphi\rangle \diff t
+ 2\Bigg( \sum_k \lambda_k^2\Bigg)\langle u,\partial_{xx}\varphi\rangle\diff t\\
& - 2 \sum_k \lambda_k \langle u, e_k\partial_x \varphi\rangle\diff W_k
- \sum_k \lambda_k \langle u, e'_k\varphi\rangle\diff W_k
\end{split}\end{equation}
Let us consider the last two series separately (in principle we have already committed an abuse by splitting the series, since we haven't yet proved its convergence, but once convergence is proven the passage is rigorous; otherwise just consider the finite approximations first, for which the splitting is legit, and push to the limit after the convergence of both series is proven). By It\^o isometry and independence of $\{W_k\}_k$, for any $t\in [0,T]$ we have
\begin{equation}\nonumber\begin{split}
\mathbb{E}\Bigg[ \bigg\vert \sum_k \lambda_k \int_0^t \langle u(s), e_k\partial_x\varphi\rangle\diff W_k(s)\bigg\vert^2 \bigg]
& = \sum_k \lambda_k^2 \int_0^t \mathbb{E}[\langle u(s)\partial_x\varphi, e_k\rangle^2]\diff s\\
& \leq \sup_k \lambda_k^2 \int_0^T \mathbb{E}\bigg[\sum_k \langle u(s)\partial_x\varphi,e_k\rangle^2\bigg]\diff s\\
& = \sup_k \lambda_k^2 \int_0^T \mathbb{E}\big[\,\vert u(s)\partial_x\varphi\vert_{L^2}^2\big]\diff s\\
& \leq \sup_k \lambda_k^2\, \Vert \partial_x\varphi\Vert_\infty^2\, \int_0^T \mathbb{E}\big[\,\vert u(s)\vert_{L^2}^2\big]\diff s
\end{split}\end{equation}
Therefore the first series grows as $\sup_k \vert \lambda_k\vert$ as in the case $d\geq 2$. For the second series with analogous calculations we find
\begin{equation}\nonumber
\mathbb{E}\Bigg[ \bigg\vert \sum_k k\lambda_k \int_0^t \langle u(s), e_k\varphi\rangle\diff W_k(s)\bigg\vert^2 \Bigg]
\leq \sup_k\big\{ k^2 \lambda_k^2\big\}\, \Vert \varphi\Vert_\infty^2 \int_0^T \mathbb{E}\big[\,\vert u(s)\vert_{L^2}^2\big]\diff s,
\end{equation}
which shows that the second series grows in norm as $\sup_k \{\vert k\lambda_k\vert\}$, which is therefore also the leading term of the overall martingale term appearing in equation \eqref{sec 5 - specific ito formulation, weak formulation}. In particular observe that
\begin{equation}\nonumber
\sum_k \lambda_k^2 \leq \sup_k \{k\lambda_k\}^2\, \sum_{k\neq 0} \frac{1}{k^2} = C\,\sup_k \{k\lambda_k\}^2,
\end{equation}
which shows that it's not possible to renormalize $W$ in such a way that in the limit the coefficient in front of $\partial_{xx}u$ survives while the martingale term disappears. The above inequality only holds in dimension $d=1$ and fails for $d$ higher, proving once again that dimension is playing a fundamental role and we can't infer for $d=1$ the same results as for $d\geq 2$. However, it's still easy to find a collection $\{\lambda_k^N, k\in\mathbb{Z}, N\in\mathbb{N}\}$ such that $\lambda^N_k = \lambda^N_{-k}$, $\lambda_0^N=0$ for all $N$ and
\begin{equation}\nonumber
\lim_{N\to\infty}\frac{\sup_k \{k^2(\lambda^N_k)^2\}}{\sum_k k^2(\lambda^N_k)^2}=0,
\qquad \lim_{N\to\infty}\frac{\sum_k (\lambda_k^N)^2}{\sum_k k^2(\lambda^N_k)^2}=0,
\end{equation}
where we are of course assuming that for fixed $N$ all the quantities appearing are finite. For a given such sequence, if we define the noises $W^N=W^N(t,x)$ as
\begin{equation}\nonumber
W^N(t,x)= \sum_k \lambda^N_k e_k(x) W_k(t)
\end{equation}
and for a fixed $\nu>0$ we define
\begin{equation}\nonumber
\varepsilon_N = 2\nu \bigg( \sum_k k^2(\lambda^N_k)^2\bigg)^{-1},
\end{equation}
then $\varepsilon_N\to 0$ as $N\to\infty$ and going through the same proof as in Theorem \ref{theorem sec 3 - main result} we obtain that any weak energy solution of equation
\begin{equation}\nonumber
\begin{cases}
\diff u^N = b\, \partial_x u^N\diff t + 2\sqrt{\varepsilon_N}\circ \diff W^N\, \partial_x u^N + \sqrt{\varepsilon_N}\circ \diff (\partial_x W^N)\, u^N\\
u^N(0)=u_0
\end{cases}
\end{equation}
will converge in probability as $N$ goes to infinity to $u$ deterministic solution of
\begin{equation}\label{sec 5 - expected deterministic pde limit}
\begin{cases}
\partial_t u = b\, \partial_x u - \nu\, u\\
u(0)=u_0
\end{cases},
\end{equation}
as long as the weak solution of \eqref{sec 5 - expected deterministic pde limit} is unique on the interval $[0,T]$. If uniqueness of \eqref{sec 5 - expected deterministic pde limit} fails, then the proof of Theorem \ref{theorem sec 3 - main result} breaks down as well; we may still however extract a (not relabelled) subsequence $u^N$ which converges weakly in $L^2(\Omega\times[0,T]\times\mathbb{T},\diff\mathbb{P}\otimes \diff t\otimes \diff x)$ to a stochastic solution of \eqref{sec 5 - expected deterministic pde limit}, usually referred to as a \textit{superposition solution}, see \cite{Amb}, \cite{Fla3} ; moreover by properties of weak convergence, this solution still satisfies the energy inequality. We have however lost the main advantage of Theorem \ref{theorem sec 3 - main result}, where the sequence $u^N$ converges to a deterministic PDE which is in principle much better posed than the approximating sequence, due to the presence of the Laplacian.

%
%

Such a result may still be seen, from the modelling point of view, as a mathematical justification of the presence of a friction term $-\nu u$ in one dimensional PDEs, as the ideal limit of the action of a suitable noise $W$ which is very irregular but of very small intensity; the coefficient $\nu$ is proportional to the product between the magnitude and the spatial irregularity of the noise (measured by its $H^1$ norm). We underline however that a noise of the form \eqref{sec 5 - linear spde, compact stratonovich formulation} has been introduced only for mathematical convenience (it allows to obtain an energy inequality for the solutions) and has not been justified from the physical point of view, nor equation \eqref{sec 5 - linear spde, compact stratonovich formulation} has been derived from first principles (namely from a Lagrangian formulation). Therefore there is still the possibility that the addition of a different multiplicative noise, with a more robust modelling justification, allows to obtain an analogue of Theorem \ref{theorem sec 3 - main result} also in dimension $d=1$.

\begin{acknowledgements}
This work stems from my master thesis (see \cite{Gal}, Section 4.4) which was developed under the supervision of Prof. David Barbato, to whom I'm deeply indebted and I want to express my gratitude. I also want to thank Prof. Franco Flandoli for the very useful discussions and for encouraging me into writing this work, as well as Prof. Dejun Luo for pointing out a mistake in early calculations and showing me a very simple and elegant way to find the explicit expression for the It\^o-Stratonovich term. I'm also grateful to Prof. Massimiliano Gubinelli for reviewing the early draft of this work and to Immanuel Zachhuber and Lorenzo Dello Schiavo for the useful suggestions for the proof of Lemma \ref{lemma sec 3 - sufficient condition for uniqueness of parabolic problem}.
\end{acknowledgements}
%
%
%



%
\end{document}